\documentclass[11pt]{amsart}

\usepackage{graphicx,caption,subcaption}

\usepackage{amssymb,amsmath,amsfonts,amsthm,amscd,amstext,mathtools}
\usepackage{epstopdf}
\usepackage{microtype}
\usepackage{mathrsfs}
\usepackage{enumitem}	
\usepackage{verbatim} 
\usepackage{url} 
\usepackage{hyperref}
\usepackage{xcolor}
\usepackage{tikz}

\usepackage{cancel}

\newcommand{\ind}{{\sf 1}}

\newcommand{\bP}{\mathbf{P}}

\newcommand{\bE}{\mathbf{E}}
\newcommand{\bbP}{\mathbb{P}}
\newcommand{\bbQ}{\mathbb{Q}}
\newcommand{\bbE}{\mathbb{E}}

\newcommand{\R}{\mathbb{R}}

\newcommand{\N}{\mathbb{N}}
\newcommand{\bbZ}{\mathbb{Z}}

\newcommand{\bVar}{\mathbf{V}\mathrm{ar}}

\newcommand{\bbA}{{\ensuremath{\mathbb A}}}

\newcommand{\cA}{{\ensuremath{\mathcal A}} }
\newcommand{\cF}{{\ensuremath{\mathcal F}} }

\newcommand{\cE}{{\ensuremath{\mathcal E}} }

\newcommand{\cL}{{\ensuremath{\mathcal L}} }

\newcommand{\cI}{{\ensuremath{\mathcal I}} }
\newcommand{\cZ}{{\ensuremath{\mathcal Z}} }

\renewcommand{\epsilon}{\varepsilon}
\renewcommand{\phi}{\varphi}

\newcommand{\gb}{\beta}

\newcommand{\eps}{\varepsilon}      

\newcommand{\gz}{\zeta}

\newcommand{\go}{\omega}

\newcommand{\gO}{\Omega}

\newcommand{\gL}{\Lambda}
\newcommand{\gs}{\sigma}

\newcommand{\bt}{\textbf{\textit{t}}}

\newcounter{cst}[section]		
\newcounter{svf}[section]		

\newtheorem{theorem}{Theorem}[section]
\newtheorem{proposition}[theorem]{Proposition}

\newtheorem{lemma}[theorem]{Lemma}
\newtheorem{claim}[theorem]{Claim}

\theoremstyle{definition}

\newtheorem{definition}[theorem]{Definition}
\newtheorem*{definition*}{Definition}
\newtheorem{assumption}[theorem]{Assumption}

\theoremstyle{remark}
\newtheorem{remark}{Remark}[section]

\newtheorem*{notation*}{Notation}

\numberwithin{equation}{section}			

\newcommand{\dd}{\mathrm{d}}		

\renewcommand{\tilde}{\widetilde}
\newcommand{\ol}{\overline}

\newcommand{\cS}{{\ensuremath{\mathcal S}}}
\newcommand{\cR}{{\ensuremath{\mathcal R}}}

\renewcommand{\tilde}{\widetilde}

\newcommand{\bbT}{\mathbb{T}}

\newcommand{\bbq}{\boldsymbol{q}}

\newcommand{\bq}{\boldsymbol{q}}

\newcommand{\cT}{{\ensuremath{\mathcal T}} }
\newcommand{\T}{{\ensuremath{\mathcal T}} }
\newcommand{\cTi}{{\ensuremath{\mathcal{T}_\infty}} }
\newcommand{\Ti}{{\ensuremath{\mathcal{T}_\infty}} }
\newcommand{\ET}{E_{\T}}
\newcommand{\PT}{P_{\T}}

\newcommand{\PTi}{P_{\Ti}}

\newcommand{\Spine}{\mathrm{Spine}}

\newcommand{\ceq}{\coloneqq}
\newcommand{\eqc}{\eqqcolon}

\title[Recurrence of the critical snake in random conductance]{Recurrence and transience of the critical random walk snake in random conductances}

\author[A. Legrand]{Alexandre Legrand}
\address{Université Claude Bernard Lyon 1, Institut Camille Jordan, UMR 5208. 43 boulevard du 11 novembre 1918, 69622 Villeurbanne Cedex, France}
\email{legrand@math.univ-lyon1.fr}
\author[C. Sabot]{Christophe Sabot}
\address{Université Claude Bernard Lyon 1, Institut Camille Jordan, UMR 5208. 43 boulevard du 11 novembre 1918, 69622 Villeurbanne Cedex, France; and Institut Universitaire de France}
\email{sabot@math.univ-lyon1.fr}
\author[B. Schapira]{Bruno Schapira}
\address{Aix-Marseille Université, Institut de Mathématiques de Marseille, UMR 7373, 39 rue F. Joliot Curie, 13453 Marseille Cedex 13, France}
\email{bruno.schapira@univ-amu.fr}


\makeatletter
\@namedef{subjclassname@2020}{%
	2020 Mathematics Subject Classification}
\makeatother

\date{}

\begin{document}
\begin{abstract}
In this paper we study the recurrence and transience of the $\bbZ^d$-valued branching random walk in random environment indexed by a critical Bienaym\'e-Galton-Watson tree, conditioned to survive. The environment is made either of random conductances or of random traps on each vertex. We show that when the offspring distribution is non degenerate with a finite third moment and the environment satisfies some suitable technical assumptions, then the process is recurrent up to dimension four, and transient otherwise. The proof is based on a truncated second moment method, which only requires to have good estimates on the quenched Green's function.
\end{abstract}

\maketitle

\section{Introduction}
In this paper we study the recurrence and transience of the $\bbZ^d$-valued branching random walk in an environment made of random conductances, indexed by a Bienaymé-Galton-Watson (BGW) tree conditioned to survive. 
While general criteria now exist for this question in case of supercritical offspring distribution, see in particular~\cite{CP07, GantertMuller,Muller08}, we focus here on the critical case, which has received much less attention so far, at the exception of~\cite{BC12,LGL16}, to our knowledge. See also~\cite{BP94,BP94b} for earlier results on this question for general tree-indexed random walks.

\subsection{Statement of the result}
Let $\bbq=(q_k)_{k\geq0}$ be a critical offspring distribution (that is $\sum_{k\geq0} kq_k=1$) such that $\sum_{k\ge 0} k^3 q_k<\infty$ and $\gs^2\ceq\sum_{k\geq0} k(k-1)q_k>0$ . In the following, let $\Ti$ denote \emph{Kesten's tree}, i.e. the critical BGW tree with offspring distribution $\bbq$ conditioned to survive (see~\cite{Dur03, Kes86} or~\cite[Ch. 12]{LP16} for details).

For $x,y\in\bbZ^d$, we write $x\sim y$ if $|x-y|=1$, where $|\cdot|$ is the usual Euclidean norm in $\R^d$. In the following, an \emph{environment} $\go\ceq(\go_{x,y})_{x,y\in\bbZ^d}$ shall denote a family of random, non-negative weights on $(\bbZ^d)^2$. In this paper we restrict ourselves to the following two types of environments.

\begin{definition}\label{def:go}
$(i)$ \emph{Random conductances:} For $x,y\in\bbZ^d$, one has $\go_{x,y}=\go_{y,x}>0$ if $x\sim y$, and $\go_{x,y}=0$ otherwise.

$(ii)$ \emph{Random traps:} There exists some (random) $\rho_x\in[0,1)$, $x\in\bbZ^d$, such that for $x,y\in\bbZ^d$,
\begin{equation}
\go_{x,y}\,=\,\begin{cases}
\rho_x/(1-\rho_x) & \text{if }x=y,\\
1/(2d) & \text{if }x\sim y,\\
0 &\text{otherwise.}
\end{cases}
\end{equation}
\end{definition}
The names ``random conductances'' and ``random traps'' used throughout this paper may be slightly abusive when compared to the rest of the literature, but we believe that they actually help the understanding of our results and what they entail. 
Depending on the nature of $\go$, we also define for $x\in\bbZ^d$,
\begin{equation}\label{eq:defpi}
\pi_\go(x)\,\ceq\,\begin{cases} \sum_{z\sim x}\go_{x,z} & \text{for random conductances},\\
(1-\rho_x)^{-1} & \text{for random traps}.
\end{cases}
\end{equation}  
Then, letting $p^\go=(p^\go_{x,y})_{x,y\in\bbZ^d}\ceq(\go_{x,y}/\pi_\go(x))_{x,y\in\bbZ^d}$ for a fixed realization $\go$, this defines the transition probabilities of a Markov chain in $\bbZ^d$, with invariant measure $\pi_\go$.

Therefore, the \emph{critical random walk snake} $\cS_\cTi:\cTi\to\bbZ^d$ in random environment $\go$ is defined by the (branching) random walk on $\bbZ^d$, indexed by Kesten's tree $\cTi$, and with transition probabilities $p^\go$. We let $\bbP$ denote the law of the environment $\go$, and $\bP$, $\bP^\go$ denote respectively the \emph{annealed} and \emph{quenched} laws of the critical snake (in particular $\bP=\bbE\bP^\go$: precise definitions are provided below).

First, we provide a 0--1 law for the recurrence of the critical random walk snake on $\bbZ^d$. This is achieved under the following assumption. In the following, $\tau_x$, $x\in\bbZ^d$, denote the shift operator, i.e. $(\tau_x\go)_{y,z}=\go_{y+x,z+x}$.

\begin{assumption}\label{assum:erg}[Stationarity and ergodicity]
The law $\bbP$ is stationary and ergodic with respect to translations $\tau_x$, $x\in\bbZ^d$.
\end{assumption}

\begin{proposition}\label{prop:01law}
Suppose that the environment is made either of random conductances or random traps, that Assumption~\ref{assum:erg} holds, and that $\bbE \pi_\go(0) <+\infty$. Then we have either:

$(i)$ The critical random walk snake is recurrent, and\[\bP(\forall\, x\in\bbZ^d, x\text{ is visited infinitely often})=1,\]or,

$(ii)$ The critical random walk snake is transient, and\[\bP(\exists\, x\in\bbZ^d, x\text{ is visited infinitely often})=0.\]
\end{proposition}

Notice that this 0--1 law under the annealed distribution $\bP$ implies the quenched 0--1 law: for any event $A$, $\bP(A)=0$ implies $\bP^\go(A)=0$ for $\bbP$-a.e. $\go\in\gO$.

Then we study the recurrence or transience of the critical random walk snake in random conductances. Before stating our main result, we introduce the following additional assumption.

\begin{assumption}\label{assum:core}[Finite range dependence]
There exists $R\in\N$ such that for any $x\in\bbZ^d$, the families of random variables $(\go_{x,y}: y\sim x)$ and $(\go_{z,z'}: |x-z|\geq R, z\sim z')$ are independent.
\end{assumption}

We now present our main theorem. In the following, we write $\go\in L^p(\bbP)$, $p\geq1$ whenever $\go_{x,y}\in L^p(\bbP)$ for all $x,y\in\bbZ^d$, $x\sim y$.
\begin{theorem}\label{thm:main}
$(i)$ Let $d\geq 5$ and $p,q\in(1,+\infty]$ such that $1/p+1/q< 2/d$. For any conductance environment $\go\sim\bbP$ satisfying Assumption~\ref{assum:erg} and such that $\go\in L^p(\bbP)$, $\go^{-1}\in L^q(\bbP)$, the critical random walk snake is transient $\bP$-a.s.. 

$(ii)$ Let $d\leq 4$. There exists a constant $p \in [1,+\infty)$ such that, for any conductance environment $\go\sim\bbP$ satisfying Assumptions~\ref{assum:erg},~\ref{assum:core} and such that $\go,\go^{-1} \in L^{p}(\bbP)$, the critical random walk snake is recurrent $\bP$-a.s..
\end{theorem}

\begin{remark}
$(i)$ Notice that $\go \in L^{p}(\bbP)$ with $p\ge 1$ implies $\bbE \pi_\go(0) <+\infty$, hence the 0--1 law from Proposition~\ref{prop:01law} holds under those assumptions.

$(ii)$ In Theorem~\ref{thm:main}.$(ii)$, the integrability assumption on $\go,\go^{-1}$ is stronger than in Theorem~\ref{thm:main}.$(i)$. We do not provide an explicit value for the constant $p$: this is further commented below.

$(iii)$ The Assumption~\ref{assum:core} in Theorem~\ref{thm:main}.$(ii)$ can be weakened significantly, in a way that covers e.g. environments $\go$ that satisfy a polynomial mixing property. This is discussed more precisely in Remark~\ref{rem:generalassumption} below.
\end{remark}

Let us briefly comment this result: Assumption~\ref{assum:core} and the integrability of $\go,\go^{-1}$ are required in~\cite{AH21} to obtain estimates on the \emph{Green's function} of the random walk in random environment, from which we deduce the theorem. However, we believe that our result holds for the random conductance environment under much weaker assumptions. 
To support that idea, we consider a random traps environment (see e.g.~\cite{BAC06, Bou92}), which is simpler to study than the random conductances environment, while preserving the ``trapping effect'' it has on the random walk. We provide an analogous result to Theorem~\ref{thm:main} under the following assumption. Recall~\eqref{eq:defpi}.

\begin{assumption}\label{assum:traps}[Bounded long-range correlations]
There exists $R_0>0$ such that
\begin{equation}\label{eq:assum:traps}
\sup\left\{\bbE[\pi_\go(x)\pi_\go(y)]\;;\, x,y\in\bbZ^d\,,\, |x-y|>R_0\right\}\,<\,+\infty\,.
\end{equation}
\end{assumption}

\begin{theorem}\label{thm:main:traps}
Let $\go$ be a random traps environment. 

$(i)$ Let $d\geq 5$, and assume $\sup_{x\in\bbZ^d}\bbE \pi_\go(x) <+\infty$. Then, the critical snake is transient $\bP$-a.s.. 

$(ii)$ Let $d\leq 4$, suppose that Assumptions~\ref{assum:erg},~\ref{assum:traps} hold and that $\bbE \pi_\go(0) <+\infty$. Then the critical snake is recurrent $\bP$-a.s..
\end{theorem}

\begin{remark}
When $d\leq2$, it is well-known that the (non-branching) random walk in random conductances or traps is almost surely recurrent (supposing additionally that $\sup_{x\in\bbZ^d}\bbE\pi_\go(x)<+\infty$ in the case of conductances). This directly implies that the critical snake is $\bP$-a.s. recurrent when $d\leq2$ under assumptions much weaker than in Theorems~\ref{thm:main}.$(ii)$,~\ref{thm:main:traps}.$(ii)$ and even Proposition~\ref{prop:01law}. This statement is presented more precisely below, see in particular Proposition~\ref{prop:rwre:rec}.
\end{remark}

As mentioned above, in Theorem~\ref{thm:main}.$(ii)$ we do not try to achieve an optimal value for $p$: indeed, we strongly believe that $p=1$ should be sufficient (similarly to Theorem~\ref{thm:main:traps}), but to our knowledge this cannot be achieved with our methods at this time.

The proofs of both Theorems~\ref{thm:main} and~\ref{thm:main:traps} rely on a second moment method, which only requires to have good estimates on the quenched Green's function. In the case of constant conductances or uniformly elliptic environment, these estimates are well known: in particular, similarly as in~\cite{LGL16}, we obtain an alternate proof to that of Benjamini and Curien~\cite{BC12}, which was based on the notion of unimodularity and mass transport techniques, together with the result of Kesten~\cite{Kesten95} on concentration of the snake in a ball. 

Concerning Theorem~\ref{thm:main:traps} we use an improved second moment method, and we kill the walk when it reaches deep traps. While we believe that a similar method could be used in the setting of Theorem~\ref{thm:main}, it would be much more technical and we have refrained to pursue in this direction. 

\subsection{Some comments and open questions}
\emph{Decay of the heat kernel.} 
Our results raise some interesting questions about the decay of the quenched heat kernel of a simple random walk in random conductance, as well as for the recurrence/transience of a critical random snake in random conductance under less restrictive hypotheses. First notice that under Assumption~\ref{assum:erg}, 
denoting by $\bP^\go_0(X_{2n}=0)$ the probability that a random walk in random conductance $\go$ starting from the origin returns to the origin in $2n$ steps, then either 
\begin{equation}\label{series.heatkernel}
\sum_{n\ge 1} n\bP^\go_0(X_{2n}=0)<+\infty,
\end{equation}
holds for $\bbP$-a.e.~$\go$, or the series in~\eqref{series.heatkernel} is infinite for $\bbP$-a.e.~$\go$. Now we can ask the following natural questions. 
\begin{enumerate}
\item We will show in Lemma~\ref{lem:trans} below, that the critical random snake is transient as soon as~\eqref{series.heatkernel} is satisfied 
for $\bbP$-a.e. $\go$, in any dimension $d\ge 1$. It is known, that this condition cannot be satisfied in dimensions $1$ and $2$, at least under the hypothesis that $\bbE [\pi_\go(0)] <+\infty$, since in this case it is known that the simple random walk, and a fortiori the critical random snake, is recurrent (see below for more details). But the question of whether or not there exists a distribution of random conductances such that~\eqref{series.heatkernel} holds in dimension $3$ or $4$ is still open, to the best of our knowledge. 
\item Likewise, the question of whether or not, in dimension $d\ge 5$,~\eqref{series.heatkernel} is always satisfied for any distribution of random conductances seems to be also open. Note that in~\cite{BBHK08}, it is shown that for bounded conductances in $d\ge 5$, one always has $\lim_{n\to \infty} n^2\bP^\go_0(X_{2n}=0)=0$, which is close to show~\eqref{series.heatkernel}, but not quite. In the other direction, it is also shown in~\cite{BBHK08} that for any $\kappa>1/d$, there exists a law of random conductances for which $n^2\bP^\go_0(X_{2n}=0)\ge C(\go)e^{-(\log n)^\kappa}$, and furthermore for any sequence $(\lambda_n)_{n\ge 0}$ increasing, converging to infinity, there exists a law of random conductances such that $\bP^\go_0(X_{2n}=0)\ge \tfrac{C(\go)}{\lambda_n n^2}$, along a subsequence. But none of these bounds  contradicts the validity of~\eqref{series.heatkernel}. 
\item Finally one can ask whether the condition $\sum_{n\ge 1} n\bP^\go_0(X_{2n}=0)=+\infty$, for $\bbP$-a.e.~$\go$, always implies recurrence of the critical random snake.
\end{enumerate}

\emph{Hitting probability of a finite BRWRE.} 
For $x\in\bbZ^d$, $d\geq1$, let $\tilde\cL(x)$ denote the local time at $x$ of a branching random walk in random environment $(\cS_{\tilde\cT},{\tilde\cT})$, where $\tilde\cT$ is a critical BGW tree \emph{not} conditioned to survive. Our second moment method can be adapted to prove that, under the assumptions of Theorem~\ref{thm:main}, one has, 
\[
\bP\big(\tilde\cL(x)>0\big)\gtrsim \begin{cases}
|x|^{-2} & \text{if }d = 3, \\ \log(|x|)^{-1}|x|^{-2} & \text{if }d=4, \\|x|^{-(d-2)} & \text{if }d\geq 5\,.
\end{cases}
\]
Matching upper bounds follow from Markov's inequality for $d \ge 5$, but would require additional arguments in lower dimension, as for the lower bounds in dimension one and two. 
In the wake of the discussion started by Benjamini and Curien~\cite[Section~3.2]{BC12}, these results would be needed to study more intricate questions on the behavior of the snake in random conductances, such as obtaining in the recurrent case the variance of the local time at 0 of a ``truncated'' snake, or the growth rate of its range. We believe that answering these questions would require several deep results (such as~\cite{Kesten95}) to be adapted from the homogeneous setting to the random environment case, so we leave them to further work.

\emph{More general random environments.} 
We believe that Theorems~\ref{thm:main} or~\ref{thm:main:traps} should also hold in more general environments, e.g. for the combination of the conductance and trap environments from Definition~\ref{def:go}. However proving this would require substantial additional
work, since one would need to adapt~\cite{AH21} (see in particular Proposition~3.1 therein, as well as~\cite[Proposition~4.7]{ADS16}) to that particular setting. Nonetheless, we do prove that the 0-1 law holds for the combined environment of traps and conductances, see~\eqref{eq:mixedenv} below.

\subsubsection*{Outline of the paper} In Section~\ref{sec:01law} we introduce some precise notation and definition for the critical random walk snake, and we prove the 0--1 law in Proposition~\ref{prop:01law} by rewriting the random process as an ergodic dynamical system. In Section~\ref{sec:estimates} we provide some Green's function estimates on (non-branching) random walks in random environment, many of them coming from~\cite{AH21}. Finally in Section~\ref{sec:transrec}, we prove both Theorems~\ref{thm:main} and~\ref{thm:main:traps}. First we prove the transience for $d\geq5$ with a direct first moment computation, then we prove the recurrence for $d\leq 4$ with the foretold second moment method.

For the sake of completeness, we also present a proof of Proposition~\ref{prop:rwre:rec} ---that is, the recurrence of the random walk in random conductances for $d\leq 2$--- in Appendix~\ref{app:dim2}.

\subsubsection*{Notation} In the remainder of this paper, $c$ denotes a constant that may change from one occurrence to another, and $c_1,c_2,\ldots$ denote constants that may change from one paragraph to another.

\section{Proof of the 0-1 law}\label{sec:01law}
In this section we introduce some notation and prove Proposition~\ref{prop:01law}. This 0-1 law is analogous to~\cite[Propositions~1.2 and~1.3]{CP07}, in the case of a super-critical BRWRE with no death: however, for the critical snake, the proof is very different. It follows from two observations: first, under our assumptions, the environment can be seen \emph{from the point of view of the particle} with an explicit change of measure $\bbQ\sim\bbP$; and second, the critical snake in random environment with law $\bbQ$ can be seen as an \emph{ergodic} dynamical system.

\subsection{Point of view of the particle}
Define
\begin{equation}\label{def:gO}
\gO\;\ceq\;\bigg\{\go\in(\R_+)^{(\bbZ^d)^2}\,;\,\forall\,x\in\bbZ^d,\,\pi_\go(x)\ceq\sum_{z\in\bbZ^d}\go_{x,z}<+\infty \bigg\}\,,
\end{equation}
and let $(\gO,\cF_\gO)$ denote the measurable space of random environment configurations on $\bbZ^d$. For the sake of generality, in this section we consider a \emph{mixed} environment of traps and conductances: more precisely, let $\bbP$ be a probability distribution on $\gO$ such that, $\bbP$-almost surely for all $x,y\in\bbZ^d$,
\begin{equation}\label{eq:mixedenv}
\go_{x,y}\,=\,\begin{cases}
\rho_x/(1-\rho_x) & \text{if }x=y,\\
\go_{y,x} >0  & \text{if }x\sim y,\\
0 &\text{otherwise,}
\end{cases}
\end{equation}
where $\go_{x,y}$, $x\sim y$ is an i.i.d. family of positive variables, and $\rho_x$, $x\in\bbZ^d$ is an i.i.d. family of random variables in $[0,1)$, independent from the $\go$'s. If $\rho_x\equiv0$, $\bbP$ is the law of a random conductance environment; and if $\go_{x,y}\equiv 1/2d$ for all $x\sim y$, it is a random trap environment. We shall prove that Proposition~\ref{prop:01law} holds for any such ``mixed'' environment which satisfies Assumption~\ref{assum:erg} and $\bbE\pi_\go(0)<+\infty$.

For $\bbP$-almost every $\go\in\gO$, the family $p^\go_{x,y}=\go_{x,y}/\pi_\go(x)$, $x,y\in\bbZ^d$, denotes the (\emph{quenched}) transition probabilities of a random walk in random environment (RWRE). Let $(X_n)_{n\geq0}$ denote the RWRE, write $\bP^\go_x$ for its quenched law in $\go\in\gO$ started from $x\in\bbZ^d$, and let $\bP_x\ceq \bbE\bP^\go_x$ be its \emph{annealed} law (when $x=0$, we may omit the subscript).

Recall that $\tau_x$, $x\in\bbZ^d$ denotes the shift operators in $\bbZ^d$. Then, starting from some fixed $\go\in\gO$, the sequence $(\tau_{X_n}\go)_{n\geq0}$ under $\bP^\go_0$ defines a Markov chain on the space of all environments $\gO$, which is called the \emph{point of view of the particle}, and its transition kernel $R$ satisfies for $f:\gO\to\R$ bounded measurable and $\go\in\gO$,
\begin{equation}\label{eq:ergodicity:tildeR}
(Rf)(\go)\,\ceq\, \sum_{z\in\bbZ^d} \frac{\go_{0,z}}{\pi_\go(0)}\, f(\tau_{z}\go)\,.
\end{equation}
We have the following classical result, which follows directly from~\cite[Lemma~2.1 and Proposition~2.3]{Bis11Rev} (see also~\cite[Theorem~1.2]{BS12book}).
\begin{proposition}\label{prop:particleviewpointmeasure}
Let $\go$ be a random environment with law $\bbP$ as in~\eqref{eq:mixedenv}. Suppose that Assumption~\ref{assum:erg} holds, and that $\bbE \pi_\go(0) <+\infty$. Then the Markov chain $(\tau_{X_n}\go)_{n\geq0}$ admits a stationary and reversible probability measure $\bbQ$, given by
\begin{equation}\label{eq:prop:particleviewpointmeasure}
\frac{\dd\bbQ}{\dd\bbP}(\go)\,:=\, \frac{\pi_\go(0)}{\bbE[\pi_\go(0)]}\,>\,0\,.
\end{equation}
In particular, $\bbP$ and $\bbQ$ are absolutely continuous with respect to one another. Moreover, $(\tau_{X_n}\go)_{n\geq0}$ is ergodic under $\bbQ$.
\end{proposition}

\subsection{Ergodicity of the snake}\label{sec:ergodicitysnake}
For $\bbP$, $\bbQ$ two probability measures on $(\gO,\cF_\gO)$, let $\bbE\eqc\bbE_\bbP$, resp. $\bbE_\bbQ$, denote the expectation under $\bbP$, resp. $\bbQ$. Let us now define a measure-preserving dynamical system on critical snakes in random environment. We first introduce some notation for the critical snake.

Let $(\bbT,\cF_\bbT)$ denote the usual measurable space of planar trees (i.e. rooted, ordered, locally finite trees). Recall that $\Ti$ denotes \emph{Kesten's tree}: more precisely, it is almost surely composed of an infinite \emph{spine}, denoted $\Spine(\Ti)\subset\Ti$, of individuals reproducing according to the \emph{mass-biased} distribution $(kq_k)_{k\geq0}$; and all other individuals reproduce according to $\bbq$. Also, let $\T$ denote an (almost surely finite) random BGW tree such that all vertices reproduce according to $\bbq$, except for the root which reproduces according to $((k+1)q_{k+1})_{k\geq0}$. As a matter of fact, $\T$ has the same law as the finite descendant tree supported by an element $u\in\Spine(\Ti)$ from the spine of Kesten's tree. Let $\PTi$ (resp. $\PT$) denote the law of $\Ti$ (resp. $\T$).

Let $\bbP$ be a probability measure on $(\gO,\cF_\gO)$ satisfying Assumption~\ref{assum:erg}. 
Let $\go\in\gO$ be a realization of the environment and $x\in\bbZ^d$: then the \emph{quenched} law $\bP^\go_x$ of the critical random walk snake $(\cS_\Ti,\Ti)$ starting from $x$ is defined as follows. Let $\Ti\sim\PTi$, and denote its root $\rho$. Then construct randomly $\cS_\Ti:\Ti\to\bbZ^d$ conditionally to $\Ti$ by induction, letting $\cS_\Ti(\rho)=x$, and for $u,v\in\Ti$ such that $v$ is a child of $u$, letting
\begin{equation}\label{eq:snake:transitions}
\bP^\go_x(\cS_\Ti(v)=z\,|\,\cS_\Ti(u)=y)\,=\,\frac{\go_{y,z}}{\pi_\go(y)}\,,\qquad y,z\in\bbZ^d\,.
\end{equation}
Moreover, conditionally on $\cS_\Ti(u)$, the random variables $\cS_\Ti(v)$, with $v$ a child of $u$, are taken independent of each other. We also define the \emph{annealed} law $\bP_x\ceq\bbE\bP^\go_x$, and for both laws we may omit the subscript $x$ when $x=0$. Notice that we use the same notation for the laws of the critical snake $(\cS_\Ti,\Ti)$ and the RWRE $(X_n)_{n\geq0}$ introduced above, but it will always be clear from context which one is considered. Furthermore, one can define similarly the (quenched or annealed) law of the tree-indexed random walk $(\cS_\T,\T)$, by letting $\cT\sim\PT$ and constructing $\cS_\cT$ by induction with~\eqref{eq:snake:transitions}; again, we abusively use the same notation for the laws.\smallskip

We now present an encoding of critical snake realizations for which we have an explicit, ergodic transformation. 
For any finite, rooted ordered tree $(\bt,\rho)\in\bbT$, one can define a bijection between $E(\bt)$ the set of edges of $\bt$, and $\{1,\ldots,\# E(\bt)\}$ (e.g. the breadth-first exploration of $\bt$); therefore, if $\cS_{\bt,\rho}(\cdot)$ denotes a realization of a $\bbZ^d$-valued random walk indexed by $\bt$, it can be rewritten as a random variable on the measured space $((\bbZ^d)^{\# E(\bt)},\cF_{\cS,\# E(\bt)})$ where $\cF_{\cS,\# E(\bt)}$ denotes the product sigma-algebra on $(\bbZ^d)^{\# E(\bt)}$. With these notation at hand, we finally define
\begin{align*}
&\bbA\,:=\\
&\quad\left\{ \left(\go, (x_i,\bt_i,\cS^{x_i}_{\bt_i,i})_{i\geq0} \right)\,\middle|\, \go\in\gO\,,\, x_0=0\,;\, \forall\,i\geq0,\,x_i\in\bbZ^d,\,\bt_i\in\bbT\text{ and } \cS^{x_i}_{\bt_i,i}\in (\bbZ^d)^{\#E(\bt_i)}\right\}\,,
\end{align*}
and we endow $\bbA$ with its product sigma-algebra $\cA$. 

Let $\go\in\gO$, let $(\cS_\Ti,\Ti)$ be a realization of the critical snake in $\go$ started from $0\in\bbZ^d$. Let $(u_i)_{i\geq0}\ceq\Spine(\Ti)$, and for $i\geq0$ let $\T_i$ be the (largest) sub-tree of $\Ti$ rooted in $u_i$ and not containing $u_{i-1}$ and $u_{i+1}$. Consider the (injective) mapping,
\begin{equation}\label{eq:pushforward}
\Phi\;:\;(\go,\cS_\Ti,\Ti)\,\mapsto\, \left(\go, (\cS_\Ti(u_i),\,\T_i,\,\cS_\Ti|_{\T_i})_{i\geq0} \right)\,\in\bbA\,,
\end{equation}
Therefore, the annealed distribution $\bP$ of the snake $(\go,\cS_\Ti,\Ti)$ induces a probability measure on $(\bbA,\cA)$, which we also denote $\bP$ abusively. Additionally, under the mapping above one has that $(X_i)_{i\geq0}\ceq(\cS_\Ti(u_i))_{i\geq0}$ is a RWRE started from 0 with transition probabilities induced by $\go$, $\T_i\sim \PT$ are i.i.d. and independent from $\go, (X_i)_{i\geq0}$; and for $i\geq0$, $\cS_{i}\ceq \cS_\Ti|_{\T_i}$ are independent walks indexed by $\T_i$ started from $X_i\in\bbZ^d$, with transition probabilities $(\go_{x,y}/\pi_\go(x))_{x,y\in\bbZ^d}$.

Let $\cR:\bbA\to \bbA$ be defined by,
\begin{equation}\label{eq:ergodicity:R}
\cR\left(\go, (x_i,\bt_i,\cS^{x_i}_{\bt_i,i})_{i\geq0} \right)\,:=\, \left(\tau_{x_1}\go\,,\, \big(x_{i+1}-x_1,\bt_{i+1},\cS^{x_{i+1}-x_1}_{\bt_{i+1},i+1}\big)_{i\geq0} \right)\,.
\end{equation}
When applied to (the image by $\Phi$ of) a critical snake realization $(\cS_\Ti,\Ti)$, $\cR$ is the transformation that re-roots it at the next vertex along the spine $u_1$, removes the former root $u_0$ and the finite tree it supported, and shifts the environment and the trajectories in $\bbZ^d$ by $X_1\ceq\cS_\Ti(u_1)$. The application $\Phi$ and the dynamical system $(\bbA,\cR)$ are illustrated in Figure~\ref{fig:dynsystem}.

\begin{figure}[ht]
\begin{center}
\begin{subfigure}[c]{.26\textwidth}
\includegraphics[width=\linewidth]{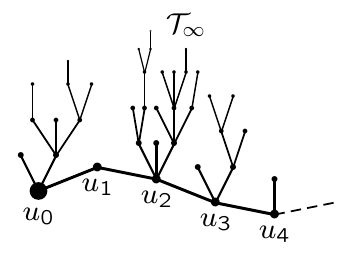}
\caption*{\vspace{-4mm}$(\go,\cS_\Ti,\Ti)$}
\end{subfigure} $\overset{\Phi}{\longrightarrow}$ 
\begin{subfigure}[c]{0.3\textwidth}
\includegraphics[width=\linewidth]{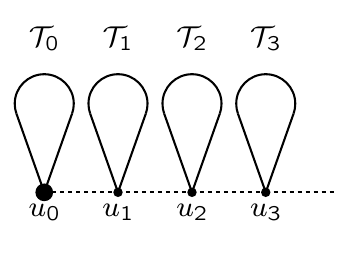}
\caption*{$\left(\go, (x_i,\cT_i,\cS_{\cT_i})_{i\geq0} \right)$}
\end{subfigure} $\overset{\cR}{\longrightarrow}$
\begin{subfigure}[c]{0.3\textwidth}
\includegraphics[width=\linewidth]{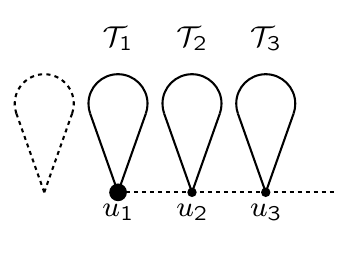}
\caption*{$\left(\tilde\go\,,\, \big(\tilde x_{i+1},\cT_{i+1},\tilde\cS_{\cT_{i+1}}\big)_{i\geq0} \right)$}
\end{subfigure}
\caption{\footnotesize Illustration of the dynamical system. An (annealed) realization of the snake is a triplet $(\go,\cS_\Ti,\Ti)$, where the tree $\Ti$ contains an infinite spine $(u_i)_{i\geq0}$ and is rooted in $u_0$, i.e. $\cS_\Ti(u_0)=0$. The application $\Phi$ maps the realization into $\bbA$, by splitting $\Ti$ into a sequence of finite trees $(\T_i)_{i\geq0}$, and defining $\cS_{\T_i}=\cS_{\Ti}|_{\T_i}$ and $x_i=\cS_{\T_i}(u_i)$. Then, the transformation $\cR:\bbA\to\bbA$ modifies this sequence by removing the first tree $\cT_0$ and shifting the environment $\go$ and the maps $\cS_{\cT_i}$, so that the new root $u_1$ is sent to $0$ in $\bbZ^d$: see~\eqref{eq:ergodicity:R} for the definitions of $\tilde\go$, $\tilde x_{i+1}$ and $\tilde\cS_{\cT_{i+1}}$, $i\geq0$.}
\label{fig:dynsystem}
\end{center}
\end{figure}

Assume that there exists a probability measure $\bbQ\sim\bbP$ which is $R$-invariant and ergodic, as in Proposition~\ref{prop:particleviewpointmeasure}. 
We may define $\tilde\bP$ a new annealed probability distribution of the random snake, where the environment has law $\bbQ$ instead of $\bbP$ and the rest of the definition is unchanged (that is, the quenched distributions $\bP(\cdot|\go)=\bP^\go(\cdot)=\tilde\bP(\cdot|\go)$ are identical $\bbP$-a.s., recall~\eqref{eq:snake:transitions}). Similarly to $\bP$, we may push it forward to $\bbA$ with the mapping~\eqref{eq:pushforward}: then, we prove that the dynamical system $(\bbA,\cA,\tilde \bP,\cR)$ is ergodic.

\begin{lemma}\label{lem:ergodicity}
Let $\bbP$ be such that Assumption~\ref{assum:erg} holds. Suppose that there exists $\bbQ$ a $R$-invariant, ergodic measure, and that $\bbP$ and $\bbQ$ are absolutely continuous with respect to one another. Then:

$(i)$ The probability measures $\bP$ and $\tilde \bP$ on $(\bbA,\cA)$ are absolutely continuous with respect to one another.

$(ii)$ The probability measure $\tilde \bP$ on $(\bbA,\cA)$ is $\cR$-invariant.

$(iii)$ The dynamical system $(\bbA,\cA,\tilde \bP,\cR)$ is ergodic.
\end{lemma}

\begin{remark}
Let us point out that this lemma also holds for non-reversible environments.
\end{remark}

\begin{proof}
These results follow from quite standard arguments regarding dynamical systems and the point of view of the particle, see e.g. the proof of \cite[Theorem~1.2]{BS12book}.

$(i)$. Recall that $\bP$ and $\tilde\bP$ have the same distributions conditionally to $\go$,  $\bbP$-a.s. by definition. Therefore one has,
\[
\frac{\dd\tilde\bP}{\dd\bP}\left(\go, (X_i,T_i,\cS_i)_{i\geq0}\right)\,=\,
\frac{\dd\bbQ}{\dd\bbP}(\go)\,>\,0\,,\qquad\text{for }\bP\text{-a.e. } \left(\go, (X_i,T_i,\cS_i)_{i\geq0}\right)\in \bbA\,.
\]

$(ii)$. Let $f:\bbA\to \R$ be a bounded measurable function. Conditioning with respect to $\go$ and $X_1$, one has,
\begin{align*}
&\tilde\bE\left[(f\circ \cR)\Big(\go, (X_i,T_i,\cS_i)_{i\geq0}\Big)\right] \hspace{-0.45pt}=\hspace{-0.45pt} \bbE_\bbQ\Bigg[\sum_{z\in\bbZ^d} \frac{\go_{0,z}}{\pi_\go(0)} \tilde\bE\Big[(f\circ \cR)\Big(\go, (X_i,T_i,\cS_i)_{i\geq0}\Big) \Big|\go, X_1=z \Big] \Bigg]\\
&\qquad =\, \bbE_\bbQ\left[\sum_{z\in\bbZ^d} \frac{\go_{0,z}}{\pi_\go(0)}\, \tilde\bE\left[f\Big(\tau_{z}\go, (X_{i+1}-z,T_{i+1},\cS_{i+1}-z)_{i\geq0}\Big) \,\middle|\,\go, X_1=z \right] \right]\\
&\qquad =\, \bbE_\bbQ\left[\sum_{z\in\bbZ^d} \frac{\go_{0,z}}{\pi_\go(0)}\, \tilde\bE\left[f\Big(\go', (X_{i},T_{i},\cS_{i})_{i\geq0}\Big) \,\middle|\,\go'=\tau_z\go\right] \right]\,,
\end{align*}
where the last equality is obtained by applying the Markov property. Recalling~\eqref{eq:ergodicity:tildeR} and that $\bbQ$ is $R$-invariant, this yields
\begin{align*}
&\tilde\bE\left[(f\circ \cR)\Big(\go, (X_i,T_i,\cS_i)_{i\geq0}\Big)\right] \\
&\qquad = \bbE_\bbQ\left[\tilde\bE\left[f\Big(\go, (X_{i},T_{i},\cS_{i})_{i\geq0}\Big) \,\middle|\,\go\right] \right]\,=\, \tilde\bE\left[f\Big(\go, (X_i,T_i,\cS_i)_{i\geq0}\Big)\right]\,,
\end{align*}
thus $\tilde \bP$ is $\cR$-invariant.

$(iii)$. Let $\cF_n:=\gs(\go,(X_i,T_i,\cS_i)_{i\leq n})$, $n\geq0$ be a filtration of $\cA$. Let $f:\bbA\to\R$ be bounded measurable such that $f\circ \cR=f$ $\tilde \bP$-a.e., and define
\[
\phi(\go)\,:=\,\tilde\bE\left[f\Big(\go, (X_i,T_i,\cS_i)_{i\geq0}\Big)\,\middle|\,\go\right]\,,\quad \go\in\gO\,.
\]
We claim that $\phi(\tau_{X_n}\go)$, $n\geq0$ is an $(\cF_n)_{n\geq0}$-martingale. Indeed,
\begin{align*}
\tilde\bE\left[f\Big(\go, (X_i,T_i,\cS_i)_{i\geq0}\Big)\,\middle|\,\cF_n\right]\,&=\, \tilde\bE\left[(f\circ \cR^{\circ n})\Big(\go, (X_i,T_i,\cS_i)_{i\geq0}\Big)\,\middle|\,\cF_n\right]\\
&=\, \tilde\bE\left[f\Big(\go', (\ol X_i,\ol T_i,\ol \cS_i)_{i\geq0}\Big)\,\middle|\,\go'=\tau_{X_n}\go\right] \,=\, \phi(\tau_{X_n}\go)\,,
\end{align*}
where the second equality follows from the Markov property, with $(\ol X_i,\ol T_i,\ol \cS_i)_{i\geq0}$ a copy of $(X_i,T_i,\cS_i)_{i\geq0}$ in the same environment. Therefore, $(\phi(\tau_{X_n}\go))_{n\geq0}$ converges $\tilde \bP$-a.e. and in $L^1(\tilde\bP)$ towards $f\Big(\go, (X_i,T_i,\cS_i)_{i\geq0}\Big)$.

In particular for $A\in\cA$ a $\cR$-invariant set, let $f:=\ind_A$ and define $\phi$ as above: then we claim that $\phi(\go)\in\{0,1\}$ for $\bbQ$-a.e $\go\in\gO$. Indeed, otherwise there would exist $[a,b]\subset\R\setminus\{0,1\}$ such that $\bbQ(\phi(\go)\in[a,b])>0$. However, Birkhoff's ergodic theorem would yield that
\[
\frac1n \sum_{k=0}^{n-1} \ind_{\{\phi(\tau_{X_n}\go)\in[a,b]\}}\,\underset{n\to+\infty}{\longrightarrow}\, \tilde \bP(\phi(\go)\in[a,b]\,|\,\cI)\,,\qquad\tilde\bP\text{-a.s. and in } L^1(\tilde\bP)\,,
\]
where $\cI$ is the sigma-field of $\cR$-invariant events in $\cA$. However, taking the expectation above we see that $\tilde\bE[\tilde\bP(\phi(\go)\in[a,b]|\cI)]=\bbQ(\phi(\go)\in[a,b])>0$, which contradicts that $\phi(\tau_{X_n}\go)\to\ind_A$ $\tilde\bP$-a.s.. We conclude that there exists $B\in\cF_\gO$ such that $\phi=\ind_B$ $\bbQ$-a.s.. Since we assumed that $A$ is $\cR$-invariant, then $B$ is $R$-invariant; and since $\bbQ$ is ergodic, this implies $\tilde\bP(A)=\bbQ(B)\in\{0,1\}$, which finishes the proof. 
\end{proof}

With Lemma~\ref{lem:ergodicity} at hand, we may finally prove the 0--1 law for the recurrence of the critical snake in $\bbZ^d$.
\begin{proof}[Proof of Proposition~\ref{prop:01law}]
Let $\go$ be a random environment with law $\bbP$ as in~\eqref{eq:mixedenv}: then Proposition~\ref{prop:particleviewpointmeasure} implies that Lemma~\ref{lem:ergodicity} holds. By Lemma~\ref{lem:ergodicity}.$(i)$, it is sufficient to prove the 0--1 law under the distribution $\tilde \bP$; then it also holds for $\bP$.

Notice that the event $A:=\{\forall\, x\in\bbZ^d, x\text{ is visited infinitely often}\}\in\cA$ is $\cR$-invariant: hence, the ergodicity of $(\bbA,\cA,\tilde\bP,\cR)$ implies that $\tilde\bP(A)\in\{0,1\}$. Let us assume that the event $B:=\{\exists\, x\in\bbZ^d, x\text{ is visited infinitely often}\}$ has positive $\tilde\bP$-probability, and show that it implies $\tilde\bP(A)>0$ (and thus $\tilde\bP(A)=1$). Writing a direct union bound, one notices that $\tilde\bP(B)>0$ implies that there exists $x_0\in\bbZ^d$ which is visited infinitely often with positive probability, i.e. $\tilde\bP(\#\cS_\Ti^{-1}(x_0)=+\infty)>0$. Let us prove that
\begin{equation}\label{eq:01law:mainineq}
\forall\,z\in\bbZ^d\,,\quad \tilde\bP(\#\cS_\Ti^{-1}(z)<+\infty\,,\,\#\cS_\Ti^{-1}(x_0)=+\infty)=0\,;
\end{equation}
then,~\eqref{eq:01law:mainineq} and another union bound imply that $\tilde\bP(A)= \tilde\bP(\#\cS_\Ti^{-1}(x_0)=+\infty)>0$, which concludes the proof.

Let us prove~\eqref{eq:01law:mainineq} under the quenched distribution $\bP^\go(\cdot)=\tilde\bP(\cdot|\go)$ for $\bbQ$-a.e. $\go$, then the result follows naturally for the annealed law $\tilde\bP=\bbE_\bbQ\bP^\go$. Let $\go\in\gO$ and $(u_i)_{i\geq0}\ceq\Spine(\Ti)$, and recall from~\eqref{eq:pushforward} that $\T_i$ denotes the finite tree supported by $u_i$. Assume that $x_0$ is visited infinitely many times, that is the set $\{(i,v)\,; i\geq 0, v\in \T_i, \cS_\Ti(v)=x_0\}$ is infinite; then notice that it admits an infinite subset $\{(i_k,v_k), k\geq0\}$ such that $i_k\neq i_\ell$ for all $k\neq \ell$. Then, each vertex $v_k\in \T_{i_k}$ is the root of a critical BGW sub-tree $\tilde \T_{v_k}\subset \T_{i_k}$ with offspring distribution $\bbq$ (except when $v_k=u_{i_k}\in\Spine(\cTi)$, in which case $v_k$ has offspring distribution $((k+1)q_{k+1})_{k\geq0}$). Moreover, by the Markov property\footnote{To be more precise, this follows from a bit of stopping line theory on Markov branching processes, see e.g. \cite{Big04, Chau91, Jag89}, we do not write the details here.}, the trees $(\tilde \T_{v_k})_k$ and the walks induced by $\cS_\Ti$ on them are independent, and have positive probability of visiting $z\in\bbZ^d$: therefore, conditionally to $\#\cS_\Ti^{-1}(x_0)=+\infty$, $z$ is visited infinitely many times with $\bP^\go$-probability 1. This finishes the proof.
\end{proof}

\section{Green's function estimates}\label{sec:estimates}
This section gathers, mostly from~\cite{AH21}, some quenched estimates on the \emph{heat kernel} and \emph{Green's function} of the RWRE in conductances or traps, which are the cornerstone of our proofs in Section~\ref{sec:transrec} below. 

For a fixed realization of $\go$, we consider the \emph{heat kernel} of a RWRE $(X_n)_{n\geq0}$ in $\bbZ^d$, that is for $x,y\in\bbZ^d$, $n\geq0$,
\begin{equation}\label{eq:def:heatkernel}
P^\go_{n}(x,y)\,\ceq\, \frac{\bP^\go_x(X_n=y)}{\pi_\go(y)}\,,
\end{equation}
and its associated \emph{Green's function},
\begin{equation}\label{eq:def:green}
g^\go(x,y)\,\ceq\, \sum_{n\geq 0} P^\go_{n}(x,y)\;\in[0,+\infty]\,.
\end{equation}
Since the random walk in random conductances or random traps is reversible, one has $\bbP$-a.s. for $x,y\in\bbZ^d$,
\begin{equation}\label{eq:green:rev}
g^\go(x,y)\,=\, g^\go(y,x)\,.
\end{equation}

Estimates on the heat kernel and Green's function are very standard in the case of the \emph{homogeneous} random walk in $\bbZ^d$, which is equivalent to letting $\go_{x,y}\ceq\ind_{x\sim y}$ $\bbP$-a.s. for all $x,y\in\bbZ^d$ in our notation. Let us denote with $P_{(\cdot)}(\cdot,\cdot)$, $g(\cdot,\cdot)$ its heat kernel and Green's function respectively. It is well-known that the homogeneous heat kernel satisfies, for some $c_1,c_2,c_3,c_4>0$,
\begin{equation}\label{eq:UE:heatkernel}
c_1 n^{-d/2}\exp\big(-c_2|x-y|^2/n \big)\;\leq\; P_{n}(x,y)\;\leq\; c_3 n^{-d/2}\exp\big(-c_4|x-y|^2/n \big)\,,
\end{equation}
for all $n\in\N$, $x,y\in\bbZ^d$ such that $|x-y|\leq n$ and $(x-y)$, $n$ have the same parity (otherwise $P_{n}(x,y)=0$). In particular when $d\geq3$, a direct computation shows that this implies for some $c_5,c_6>0$ and all $x,y\in\bbZ^d$ that,
\begin{equation}\label{eq:UE:green}
c_5\,(1+|x-y|^{2-d})\,\leq\, g(x,y)\,\leq\, c_6\,(1+|x-y|^{2-d})\,,
\end{equation}
and $g(x,y)=+\infty$ if $d\leq 2$.

\begin{remark}\label{rem:UEestimates}
It follows from~\cite{Del99} that those estimates also hold uniformly $\bbP$-a.s. for the heat kernel and Green's function of a RWRE in a \emph{uniformly elliptic} environment; that is, if there exists $\eps>0$ such that $\bbP(\go_{x,y}/\pi_\go(x)\geq \eps)=1$ for all $x\sim y$.
\end{remark}

\subsection{RWRE with random traps} 
We have the following.

\begin{proposition}\label{prop:green:traps}
Let $\go$ be a random traps environment. Then, one has $\bbP$-a.s. for all $x,y\in\bbZ^d$,
\begin{equation}
    g^\go(x,y)\,=\, g(x,y)\,.
\end{equation}
\end{proposition}
\begin{proof}
Recall from~\eqref{eq:def:green} that, for $\bbP$-a.e. $\go$, $\pi_\go(y)g^\go(x,y)$ is the expected (quenched) local time in $y$ of a RWRE $(X_n)_{n\geq0}$, with $\pi_\go(y)=(1-\rho_y)^{-1}$ by definition~\eqref{eq:defpi}. Define recursively $\tau_0=0$ and for $n\geq0$,
\[
\tau_{n+1}\,:=\, \inf\{k>\tau_n, X_k\neq X_{\tau_n}\}\,.
\]
Then, $(Y_n)_{n\geq0}\ceq(X_{\tau_n})_{n\geq0}$ is exactly a homogeneous, nearest-neighbour random walk in $\bbZ^d$. Moreover, it is clear that for $n\geq0$, conditionally to $X_{\tau_n}=y\in\bbZ^d$, $\tau_{n+1}-\tau_n$ is a geometric random variable with success probability $(1-\rho_{y})$. Therefore, one has for $\bbP$-a.e. $\go$,
\begin{align*}
g^\go(x,y)\,=\, \frac1{\pi_\go(y)}\bE_x^\go\Bigg[\sum_{k\geq0} \ind_{\{X_k=y\}}\Bigg]\,=\, \frac1{\pi_\go(y)}\bE_x^\go\Bigg[\sum_{n\geq0} (\tau_{n+1}-\tau_n)\ind_{\{Y_n=y\}}\Bigg]\,=\,  \sum_{n\geq0} P_n(x,y)\,,
\end{align*}
which concludes the proof.
\end{proof}

\subsection{RWRE with random conductances}
In~\cite{AH21}, the authors provide estimates on the heat kernel and Green's function (for $d\geq3$) for the RWRE in random conductances. 
We produce some of their results here, where we mostly kept the same notation for the sake of clarity. Let us mention that all these were formulated for the continuous time process called the \emph{constant speed random walk} (CSRW), but they can straightforwardly be extended to the discrete time setting, see e.g.~\cite[Sect.~1.2]{Bis11Rev}.

\begin{remark}\label{rem:generalassumption}
In order to lighten the presentation, the authors of the present paper decided to formulate their results, notably Theorem~\ref{thm:main}.$(ii)$, in the framework of Assumption~\ref{assum:core} (finite range dependence). However, in~\cite{AH21} the authors consider a much more general setting given by~\cite[Assumption~1.3]{AH21}, which covers notably environments with finite range dependence, with non-positive correlations or with a polynomial mixing property. We claim that the following two theorems, as well as all statements from the present paper (including Theorem~\ref{thm:main}.$(ii)$), also hold when Assumption~\ref{assum:core} is replaced by the much more general~\cite[Assumption~1.3]{AH21}, with no change to the proofs or formulae. We refer to~\cite{AH21} for some motivations to their more general assumption, and several examples of classical models which are covered by them. 
\end{remark}

In the following, we always assume $\go$ is made of random conductances. Recall that we write abusively $\go\in L^p(\bbP)$, $\go^{-1}\in L^q(\bbP)$ if one has $\go_{x,y}\in L^p(\bbP)$ and $(\go_{x,y})^{-1}\in L^q(\bbP)$ for all $x\sim y$.

\begin{theorem}[Upper bounds from~\cite{AH21}]
\label{thm:AH21:UB}
Suppose that Assumption~\ref{assum:erg} holds. Let $p,q\in(1,+\infty]$ be such that $1/p+1/q< 2/d$, and assume $\go\in L^p(\bbP)$, $\go^{-1}\in L^q(\bbP)$. 
 
$(i)$ There exists $c_1,c_2>0$ and a random variable $N_1(x)=N_1(x,\go,p,q)<+\infty$, $x\in\bbZ^d$, such that for $\bbP$-a.e. $\go$ and all $x\in\bbZ^d$, one has
\begin{equation}\label{eq:thm:AH21:UB:Lp}
\sup_{m\geq N_1(x)} m^{-d}\sum_{z:|z-x|\leq m}\pi_\go(z)^p \;\leq\; c_1 \bbE[\pi_\go(0)^p]\,,
\end{equation}
and, for all $n\geq N_1(0)^2$ and $\bbP$-a.e. $\go$, one has
\begin{equation}\label{eq:thm:AH21:UB:P}
P^\go_{n}(0,0)\,\leq\, c_2\,n^{-d/2}\,.
\end{equation}

$(ii)$ If $d\geq 3$, there exists $c_3,c_3'>0$ such that, for $|x-y|\geq N_2(x)\ceq c_3' N_1(x)^2$ and $\bbP$-a.e. $\go$, one has
\begin{equation}\label{eq:thm:AH21:UB:g}
g^\go(x,y)\,\leq\, c_3 |x-y|^{2-d}\,.
\end{equation}

$(iii)$ Additionally, suppose that Assumption~\ref{assum:core} holds, and let $\gz>0$. 
There exists $p_0=p_0(\gz)\in(0,+\infty)$ such that, as soon as $\go, \go^{-1}\in L^{p_0}(\bbP)$, then $N_1(x)$ satisfies for some $c_4>0$,
\begin{equation}\label{eq:thm:AH21:N1tail}
\bbP(N_1(x) > n)\,\leq \, c_4 n^{1-d\gz}\,,\quad\forall\,n\in\N\,.
\end{equation}
Furthermore, for $d\geq3$ and all $x,y\in\bbZ^d$, one has
\begin{equation}\label{eq:thm:AH21:greenLp}
g^\go(x,y)\in L^\gb(\bbP)\,,\qquad\forall\,0\leq\gb<(d\gz-1)(p_0-1)/(2p_0)\,.
\end{equation}
\end{theorem}
\begin{proof}
The upper bound~\eqref{eq:thm:AH21:UB:Lp} is due to Assumption~\ref{assum:erg} and the spatial ergodic theorem (see also~\cite[(1.1)]{AH21}), and~\eqref{eq:thm:AH21:UB:P} is given by~\cite[Theorem~1.2]{AH21}. The claim~\eqref{eq:thm:AH21:UB:g} is directly stated in~\cite[Theorem~1.6]{AH21}. Regarding the tail estimate on $N_1$, it is a consequence of~\cite[Proposition~2.2]{AH21} and~\cite[Lemma~2.6]{AH21}. Finally, the integrability of $g^\go(x,y)^\gb$ is stated in~\cite[(4.3)]{AH21}.
\end{proof}

\begin{theorem}[Lower bounds from~\cite{AH21}]
\label{thm:AH21:LB}
Let $d\geq 3$, $\gz>0$, and suppose that Assumptions~\ref{assum:erg} and~\ref{assum:core} holds. 
There exists $p_0\in(0,+\infty)$ and a random variable $N(x)=N(x,\go,p_0)\geq N_1(x)$ such that, as soon as $\go,\go^{-1}\in L^{p_0}(\bbP)$, the following holds:

$(i)$  There exists $c_5>0$ such that, for $|x-y|\geq N(x)$ and $\bbP$-a.e. $\go$,
\begin{equation}\label{eq:thm:AH21:LB:g}
g^\go(x,y)\,\geq\, c_5 |x-y|^{2-d}\,.
\end{equation}

$(ii)$ $N(x)$ satisfies for some $c_{6}>0$,
\begin{equation}\label{eq:thm:AH21:Ntail}
\bbP(N(x) > n)\,\leq \, c_{6} n^{-d(\gz-1)+2}\,,\quad\forall\,n\in\N\,.
\end{equation}
\end{theorem}
\begin{proof}
The claim~\eqref{eq:thm:AH21:LB:g} is the content of~\cite[Theorem~1.6]{AH21}. Regarding the tail estimate on $N$, it is obtained from~\cite[Proposition~2.2]{AH21} combined with~\cite[(3.3)]{AH21} in the proof of~\cite[Theorem~1.4]{AH21}.
\end{proof}

\subsection{RWRE constrained in a large box}
For $m\geq0$, define $\gL_m\ceq\bbZ^d\cap[-m,m]^d$ and $\partial\gL_m\ceq \gL_{m+1}\setminus\gL_m$. 
Recall~\eqref{eq:def:heatkernel} and~\eqref{eq:def:green}. For $\go$ fixed, $x,y\in\bbZ^d$ and $n,m\geq0$, we define the heat kernel of the walk \emph{constrained} in the box of size $m$ by,
\begin{equation}\label{eq:def:heatkernel:ball}
P^{m}_{n}(x,y)\,\ceq\, \frac{\bP^\go_x(X_n=y\,;\,\forall \,s\leq n, X_s\in \gL_m)}{\pi_\go(y)}\,,
\end{equation}
(we omit the superscript $\go$ to lighten notation) and $g^{m}(x,y)\ceq \sum_{n\geq 0} P^{m}_{n}(x,y)$ the associated Green's function. Notice that $g^{m}$ is symmetric as in~\eqref{eq:green:rev}, and that,
\[P^{m}_{n}(x,y)\leq P^\go_n(x,y)\,,\quad g^{m}(x,y)\leq g^\go(x,y)\,,
\]
for $n,m\geq0$, $x,y\in\bbZ^d$ and for $\bbP$-a.e. $\go$. We provide a lower bound on $g^{m}(x,y)$ when $m\gg |x-y|$ in the following statement, both for random conductances and random traps.

\begin{lemma}\label{lem:constrainedgreen}
$(i)$ Let $\go$ be a random conductances environment. Suppose Assumptions~\ref{assum:erg} and~\ref{assum:core} hold. For $n,m\in\N$, consider the event
\begin{equation}\label{eq:lem:constrainedgreen:gORn}
\gO^{m}_{n}\,\ceq\,\left\{\go\,;\, \forall z\in\gL_m\,,\, N_2(z)\leq n \right\}\,.
\end{equation}
Then $\bbP(\gO^{Kn}_{n})\to1$ as $n\to+\infty$ for all $K>0$. Moreover, under the assumptions of Theorem~\ref{thm:AH21:LB}, there exists $c,K>0$ such that, for all $n\in\N$, $\go\in\gO^{Kn+1}_{n}$ and $x,y\in\gL_n$ satisfying $|x-y|\geq N(x) \vee (n/2)$, one has
\begin{equation}\label{eq:lem:constrainedgreen:LB}
g^{Kn}(x,y)\,\geq\, c\,|x-y|^{2-d}\,.
\end{equation}
$(ii)$ Let $\go$ be a random traps environment. Then, there exists $c,K>0$ such that~\eqref{eq:lem:constrainedgreen:LB} holds $\bbP$-a.s. uniformly in $n\in\N$ and $x,y\in\gL_n$ with $|x-y|\geq n/2$.
\end{lemma}

\begin{proof}
We first prove $(i)$. The first statement is straightforward: for $K>0$, one deduces from a union bound that,
\[
\bbP((\gO^{Kn}_{n})^c)\,\leq (2Kn)^d \, \bbP(N_2(0)>n)\,,
\]
for some $c>0$. Assuming $\gz$ is taken sufficiently large in~\eqref{eq:thm:AH21:N1tail}, this goes to 0 as $n\to+\infty$. Then, let $\tau_m\ceq \inf\{s\geq1, X_s\in \partial\gL_m\}$ for $m\geq0$. One has,
\begin{align*} g^\go(x,y) \,=\, g^{Kn}(x,y) + \frac1{\pi_\go(y)}\bE^\go_x\Bigg[ \sum_{s\geq \tau_{Kn}} \ind_{\{X_s=y\}}\Bigg] 
\,\leq\, g^{Kn}(x,y) + \sup_{z\in\partial\gL_{Kn}} g^\go(z,y)\,,
\end{align*}
where we used the Markov property. Taking $x,y\in\gL_n$ with $|x-y|\geq N(x)\vee (n/2)$ and $K$ sufficiently large, one deduces from~\eqref{eq:thm:AH21:UB:g} and~\eqref{eq:thm:AH21:LB:g} the desired result for all $\go\in \gO^{Kn+1}_{n}$, finishing the proof of~\eqref{eq:lem:constrainedgreen:LB}.

Regarding $(ii)$, it is deduced immediately from Proposition~\ref{prop:green:traps},~\eqref{eq:UE:green} and a similar argument (we leave the details to the reader).
\end{proof}

\section{Recurrence and transience of the critical snake}\label{sec:transrec}
We now present the proofs of recurrence or transience of the critical snake in random environment, using the Green's function estimates which we presented in Section~\ref{sec:estimates}. In order to lighten upcoming formulae, we write with an abuse of notation $|0|^{2-d}\ceq1$ for $0\in\bbZ^d$ in all the series computations below.

\subsection{Transience} Let $d\geq5$. Define,
\begin{equation}\label{eq:def:Linfty0}
\cL_\infty(0)\,\ceq\,\frac{\#\cS_{\Ti}^{-1}(\{0\})}{\pi_\go(0)}\,,
\end{equation}
the renormalized local time at $0$ of the (infinite) critical snake $(\cS_\Ti,\Ti)$ defined in Section~\ref{sec:ergodicitysnake}. 

Let $\go\in\gO$. Let $\ell(x)$ denote the local time in $x\in\bbZ^d$ of a RWRE $(X_n)_{n\in\N}$ indexed by $\Spine(\Ti)\equiv \N$ and started from $0\in\bbZ^d$. Then, let $\T^i_x$, $i\in\N$, $x\in\bbZ^d$ be independent, \emph{finite} BGW trees with offspring distribution $\bbq=(q_k)_{k\geq0}$, except for the root which has offspring distribution $((k+1)q_{k+1})_{k\geq0}$. Let $\cS^i_x$ be independent RWRE in $\bbZ^d$ started from $x$, indexed by $\T^i_x$ in the environment $\go$. Write $L^i_{x,y}\ceq\#(\cS^i_x)^{-1}(y)$, $i\in\N$, $x,y\in\bbZ^d$ the local time in $y$ of the BRWRE $(\cS^i_x,\T^i_x)$ started from $x$. Then, recalling the mapping~\eqref{eq:pushforward}, one notices that, under $\bP^\go_0$ for $\bbP$-a.e. $\go$,
\begin{equation}\label{eq:linfty}
\cL_\infty(0)\,\overset{(d)}{=}\,\frac1{\pi_\go(0)}\sum_{x\in\bbZ^d}\sum_{i=1}^{\ell(x)} L^i_{x,0}\,.
\end{equation}
Some standard computations give that, for $x\in\bbZ^d$ and $i\in\N$,
\begin{align}\label{eq:lem:Xn:Lix:1stmom}
\bE^\go[\ell(x)]\,&=\,\sum_{n\geq0}\bP^\go_0(X_n=x)\,=\, \pi_\go(x)g^\go(0,x)\,,\\
\label{eq:lem:Xn:Lix:1stmom:bis}
\text{and}\qquad\bE^\go[L^{i}_{x,0}]\,&=\,\gs^2\sum_{n\geq0}\bP^\go_x(X_n=0)\,=\, \gs^2\pi_\go(0)g^\go(x,0)\,,
\end{align}
where we recall that $\gs^2=\sum_{k\geq0}k(k-1)q_k$, with $\bbq$ the critical BGW reproduction law. Recall the definitions of $P^\go_n(x,y)$, $g^\go(x,y)$ from~(\ref{eq:def:heatkernel}--\ref{eq:def:green}): hence, one has $\bbP$-a.s. that,
\begin{equation}\label{eq:Linfty0:1stmom}
\bE^\go[\cL_\infty(0)]
\;=\; \frac{\gs^2}{\pi_\go(0)} \sum_{x\in\bbZ^d}\sum_{k,\ell\geq0} \bP^\go_0(X_k=x)\bP^\go_x(X_\ell=0)
\;=\; \gs^2\sum_{n\geq0} (n+1) P^\go_{n}(0,0)\,.
\end{equation}

We provide the following statement, which holds for \emph{any} random, \emph{elliptic} environment $\go\in\gO$, (in particular, it reaches beyond the framework of Proposition~\ref{prop:01law} or Theorem~\ref{thm:main}).
An environment is called elliptic if $\bbP(\go_{x,y}>0)=1$ for all $x\sim y$.

\begin{lemma}\label{lem:trans}
Consider a random, elliptic environment $\go\in\gO$ and let $\bbP$ denote its law. Assume that $\sum_nnP^\go_{n}(0,0)<+\infty$ for $\bbP$-a.e. $\go$. Then the critical random walk snake is transient with $\bP$-probability one.
\end{lemma}

\begin{proof}[Proof of Lemma~\ref{lem:trans}]
It follows from~\eqref{eq:Linfty0:1stmom} that, for $\bbP$-a.e. $\go$, one has $\cL_\infty(0)<+\infty$, $\bP^\go$-a.s.; hence $0$ is visited finitely many times by the critical snake with $\bP^\go$-probability one for $\bbP$-a.e. $\go$. Then, mimicking the arguments from the proof of Proposition~\ref{prop:01law}, if there existed $x_0\in\bbZ^d$ which is visited infinitely many times with positive $\bP^\go$-probability, this would also hold for $0$, yielding a contradiction.
\end{proof}

\begin{proof}[Proof of Theorem~\ref{thm:main}.$(i)$]
Recall~\eqref{eq:thm:AH21:UB:P}: letting $p,q\in(1,+\infty]$ such that $1/p+1/q< 2/d$, and assuming $\go\in L^p(\bbP)$, $\go^{-1}\in L^q(\bbP)$, one has for $d\geq5$,
\[
\sum_{n\geq
0}nP^\go_{n}(0,0)\,\leq\, N_1(0)^4 + c\sum_{n>N_1(0)^2}n^{1-d/2}\,<\,+\infty\,,\quad\bbP\text{-a.s.},
\]
which concludes the proof by Lemma~\ref{lem:trans}.
\end{proof}

\begin{proof}[Proof of Theorem~\ref{thm:main:traps}.$(i)$]
Recall~(\ref{eq:lem:Xn:Lix:1stmom}--\ref{eq:Linfty0:1stmom}) and Proposition~\ref{prop:green:traps}. In particular, one has $\bbP$-a.s.,
\[
\bE^\go[\cL_\infty(0)]\,=\, \gs^2\sum_{x\in\bbZ^d} \pi_\go(x)g^\go(0,x)g^\go(x,0)\,\leq\, c\sum_{x\in\bbZ^d} |x|^{2(2-d)}\pi_\go(x)\,,
\]
where we used the standard estimate~\eqref{eq:UE:green}. Since we assumed $\sup_{x\in\bbZ^d} \bbE\pi_\go(x)<+\infty$, this yields,
\[
\bbE\bE^\go[\cL_\infty(0)]\,\leq \, c \sum_{x\in\bbZ^d} |x|^{2(2-d)}\bbE\pi_\go(x)\,=\,c \sum_{x\in\bbZ^d} |x|^{2(2-d)}\,.
\]
For $d\geq5$, the latter series converges. Therefore, one has $\bE^\go[\cL_\infty(0)]<+\infty$ $\bbP$-a.s., which yields the expected result by Lemma~\ref{lem:trans}.
\end{proof}

\subsection{Recurrence for conductances}
In this section we prove the recurrence of the critical snake for $d\leq4$. For $m\in\N$, we let $\gL_m\ceq\bbZ^d\cap[-m,m]^d$ and $\partial\gL_m\ceq\gL_{m+1}\setminus\gL_m$ throughout this section. 

We first consider the case $d\leq2$, and reproduce the following classical result. Since the critical snake $\cS_\Ti$ restricted to $\Spine(\Ti)$ is exactly a RWRE, notice that it directly implies the recurrence of the critical snake for $d\leq2$.
\begin{proposition}\label{prop:rwre:rec}
Let $d\leq2$, and assume $\go$ is a random conductance environment on $\bbZ^d$ such that $\sup_{x\sim y}\bbE\go_{x,y}<+\infty$. Then the RWRE is recurrent $\bP$-a.s..
\end{proposition}

This result is standard, however for the sake of completeness we provide a proof in Appendix~\ref{app:dim2}.\smallskip

We now assume that $d\in\{3,4\}$. We use the Green's function estimates from Section~\ref{sec:estimates} to prove the recurrence in that case. Throughout this section, we let $p_0>1$ be large and assume $\go,\go^{-1}\in L^{p_0}(\bbP)$ so that all results from Section~\ref{sec:estimates} hold, with $\gz$ taken quite large in~\eqref{eq:thm:AH21:N1tail} and~\eqref{eq:thm:AH21:Ntail} ($\gz>d$ is enough). Let us recall that we are \emph{not} trying to obtain an optimal value for $p_0$ in this proof.
Recall~(\ref{eq:def:Linfty0}--\ref{eq:Linfty0:1stmom}). To prove Theorem~\ref{thm:main}.$(ii)$, it is enough to show that
\begin{equation}\label{eq:loctime:infinite}
\bP(\cL_\infty(0)=+\infty)\,>\,0\,,
\end{equation}
then the result follows from the 0--1 law, see Proposition~\ref{prop:01law}.

Let us prove~\eqref{eq:loctime:infinite} with a second moment method. Let us consider the BRWRE indexed by $\Ti$ which is \emph{constrained} to a box $\gL_m$, $m\in\N$ ---that is, particles from the critical snake $(\cS_\Ti,\Ti)$ that leave $\gL_m$ are removed from the process (even the particles from the spine). Let $\ell^m(x)$ denote the local time in $x\in\bbZ^d$ of the RWRE indexed by $\Spine(\Ti)$ constrained to $\gL_m$; and let $L^{i,m}_{x,y}$, $i\in\N$, $x,y\in\bbZ^d$ be the local time in $y$ of $(\cS^i_x,\T^i_x)$ constrained to $\gL_m$. Then, letting
\begin{equation}
\cL_m\,\ceq\,\frac1{\pi_\go(0)}\sum_{x\in \gL_m} \sum_{i=1}^{\ell^m(x)} L^{i,m}_{x,0}\,,
\end{equation}
one has that $\cL_m$ has same law as the renormalized local time in $0$ of the critical snake constrained to $\gL_m$. Moreover, one has $\bP^\go$-a.s. that $\cL_m$ is non-decreasing in $m$ and converges to $\cL_\infty(0)$ as $m\to+\infty$, for $\bbP$-a.e. $\go$ (recall \eqref{eq:linfty}). We have the following moment estimates on $\cL_m$.
\begin{lemma}\label{lem:Xn:1stmom}
Let $d\in\{3,4\}$. As $m\to+\infty$, one has,
\[
\bE[\cL_m]\,\asymp\, \sum_{x\in \gL_m} |x|^{2(2-d)} \asymp
    \begin{cases}
    m & \qquad\text{if }d=3,\\
    \log m & \qquad\text{if }d=4.
    \end{cases}
\]
\end{lemma}

\begin{lemma}\label{lem:Xn:2ndmom} Let $d\in\{3,4\}$. There exists $c_2>c_1>1$ such that, for all $m\in\N$, one has,
\[
\bE[\cL_m^2]\,\leq\, c_1 \,\bbE\left[\bE^\go[\cL_m]^2\right] \,\leq\, c_2 \,\bE[\cL_m]^2\,.
\]
\end{lemma}
These two lemmas immediately imply~\eqref{eq:loctime:infinite}. Indeed, one deduces from the Paley-Zygmund inequality that, for all $m\in\N$,
\[
\bP\left(\cL_\infty(0)\geq \tfrac12\bE[\cL_m]\right)\,\geq\,\bP\left(\cL_m\geq \tfrac12\bE[\cL_m]\right)\,\geq\, c\,,
\]
for some uniform $c>0$. Since one has $\bE[\cL_m]\to+\infty$ as $m\to+\infty$ by Lemma~\ref{lem:Xn:1stmom}, this yields~\eqref{eq:loctime:infinite}, and therefore Theorem~\ref{thm:main}.$(ii)$ by the 0--1 law.

We turn to the proofs of the lemmas. Recall the definition of $g^m(\cdot,\cdot)$ from~\eqref{eq:def:heatkernel:ball}.

\begin{proof}[Proof of Lemma~\ref{lem:Xn:1stmom}]
Similarly to~(\ref{eq:lem:Xn:Lix:1stmom}--\ref{eq:Linfty0:1stmom}), one has for $m\in\N$,
\begin{equation}\label{eq:lem:Xn:Lix:1stmom:n}
\bE^\go[\ell^m(x)]\,=\,\pi_\go(x)g^m(0,x)\,,\qquad\text{and}\qquad\bE^\go[L^{i,m}_{x,0}]=\gs^2\pi_\go(0)g^m(x,0)\,,
\end{equation}
and thus,
\begin{equation}\label{eq:lem:Xn:1stmom}
\bE^\go[\cL_m] \,=\, \frac1{\pi_\go(0)}\sum_{x\in \gL_m} \bE^\go[\ell^m(x)] \bE^\go[L^{1,m}_{x,0}] \,=\, \gs^2 \sum_{x\in \gL_m} \pi_\go(x) g^m(0,x)^2 \,,
\end{equation}
where we also used that $g^m(\cdot,\cdot)$ and $g^\go(\cdot,\cdot)$ are symmetric (recall~\eqref{eq:green:rev}).\smallskip

\noindent\emph{Lower bound.} Let $K>0$ and $\gO^{Kn+1}_{n}\subset\gO$ as in Lemma~\ref{lem:constrainedgreen}, $N(0)$ as in Theorem~\ref{thm:AH21:LB}, and let us compute a lower bound on $\bE[\cL_{Kn}]$. Applying Lemma~\ref{lem:constrainedgreen}, there exists $c>0$ such that,
\begin{align}
\notag \bE[\cL_{Kn}] \,&\geq\,
\gs^2 \bbE\left[\ind_{\gO^{Kn+1}_{n}}(\go)\sum_{x\in \gL_n} \big(g^{Kn}(0,x)\big)^2 \pi_\go(x)\right] \\
\label{eq:lem:Xn:1stmom:LB}
&\geq\, c\sum_{\frac n2\leq |x|\leq n} |x|^{2(2-d)} \bbE\left[\pi_\go(x) 1_{\{|x|\geq N(0)\}}\ind_{\gO^{Kn+1}_{n}}(\go)\right]\,,
\end{align}
where, in the first inequality, we restricted the sum to $\gL_n\subset\gL_{Kn}$ in~\eqref{eq:lem:Xn:1stmom}. It remains to show that the latter expectation is bounded from below uniformly in $\frac n2\leq |x|\leq n$ for $n$ sufficiently large; then the result follows from standard estimations of the series $\sum_x |x|^{2(2-d)}$. One deduces from the Cauchy-Schwarz inequality and the assumption that $\bbP$ is stationary that, for $\frac n2\leq |x|\leq n$,
\begin{align*}
\bbE\left[\pi_\go(x) 1_{\{|x|<N(0)\}}\right]
\;&\leq\;\bbE\big[\pi_\go(0)^2\big]^{1/2} \bbP\big(N(0)>n/2\big)^{1/2}\,,\\
\text{and}\qquad \bbE\left[\pi_\go(x) \ind_{\gO\setminus\gO^{Kn+1}_{n}}(\go)\right]
\;&\leq\;\bbE\big[\pi_\go(0)^2\big]^{1/2} \bbP\big(\gO\setminus\gO^{Kn+1}_{n}\big)^{1/2}\,.
\end{align*}
Provided that $p_0\geq2$, the two terms above go to zero as $n\to+\infty$ uniformly in $\frac n2\leq |x|\leq n$. Moreover, by assumption $\bbE[\pi_\go(x)]$ is a positive constant uniform in $x\in\bbZ^d$, therefore a union bound yields immediately that the expectation in~\eqref{eq:lem:Xn:1stmom:LB} is bounded from below uniformly in $\frac n2\leq |x|\leq n$ for $n$ sufficiently large, finishing the proof of the lower bound.\smallskip

\noindent\emph{Upper bound.} Since $g^m(0,x)\leq g^\go(0,x)$, one deduces from~\eqref{eq:lem:Xn:1stmom}, \eqref{eq:green:rev} and~\eqref{eq:thm:AH21:UB:g},
\begin{equation}\label{eq:lem:Xn:1stmom:UB}
\bE[\cL_m] \,\leq\, c\,\bbE\left[\sum_{|x|\leq N_2(0)} g^\go(x,0)^2\pi_\go(x)\right] + c\,\bbE\left[\sum_{N_2(0)\leq|x|\leq 2m} |x|^{2(2-d)} \pi_\go(x)\right].
\end{equation}
Since $\bbE\pi_\go(0)<+\infty$ and $\bbP$ is stationary, the second term is lower than $c\sum_{|x|\leq 2m}|x|^{2(2-d)}$ for some constant $c>0$ and all $m\in\N$; hence it suffices to prove that the first term is bounded. Noticing that $g^\go(x,0)\leq g^\go(0,0)$ for all $x\in\bbZ^d$ and $\bbP$-a.e. $\go$, one deduces from H\"older's inequality for some $p\in(1,p_0)$ that,
\begin{equation}\label{eq:lem:Xn:1stmom:Holder}
\bbE\left[\sum_{|x|\leq N_2(0)} g^\go(x,0)^2\pi_\go(x)\right]\,\leq\, \bbE\left[g^\go(0,0)^{2p/(p-1)}\right]^{(p-1)/p}\bbE\Bigg[\bigg(\sum_{|x|\leq N_2(0)} \pi_\go(x)\bigg)^p\Bigg]^{1/p}.
\end{equation}
Recalling~\eqref{eq:thm:AH21:greenLp} and assuming that $\gz, p, p_0$ are taken sufficiently large, the first factor is a finite constant. Moreover, we have the following.
\begin{claim}\label{claim:pi:Lp}
Let $p,r\geq1$ and assume $pr\leq p_0$. Then there exists $c>0$ such that,
\begin{equation}
\Bigg\|\sum_{|x|\leq N_2(0)} \pi_\go(x)^r\Bigg\|_{L^p(\bbP)} \,\leq\, c\,\big\|\pi_\go(0)^r\big\|_{L^p(\bbP)}\|N_2(0)^{d}\|_{L^p(\bbP)}\,.
\end{equation}
\end{claim}
In particular, this implies that the second factor in~\eqref{eq:lem:Xn:1stmom:Holder} is finite (assuming again that $\gz$ is large enough in~\eqref{eq:thm:AH21:N1tail}), which finishes the proof of the upper bound and Lemma~\ref{lem:Xn:1stmom}.
\end{proof}
\begin{proof}[Proof of Claim~\ref{claim:pi:Lp}]
By H\"older's inequality, one has
\begin{equation}\label{eq:pi:Holder}
\Bigg(\sum_{|x|\leq N_2(0)} \pi_\go(x)^r\Bigg)^p \,\leq\, c\, N_2(0)^{d(p-1)}\sum_{|x|\leq N_2(0)} \pi_\go(x)^{rp}
\end{equation}
for some $c>0$. Recalling~\eqref{eq:thm:AH21:UB:Lp} and taking the expectation, this finishes the proof of the claim.
\end{proof}

We now turn to the proof of the second moment estimates in Lemma~\ref{lem:Xn:2ndmom}.

\begin{proof}[Proof of Lemma~\ref{lem:Xn:2ndmom}]
We start by proving the second inequality. Recollecting~\eqref{eq:lem:Xn:1stmom}, recalling that $g^m(\cdot,\cdot)$ is symmetric, that $g^m(x,y)\leq g^\go(x,y)$, and that $g^\go(x,0)\leq g^\go(0,0)$ for all $x,y,m$ and $\bbP$-a.e. $\go$, one deduces from~\eqref{eq:thm:AH21:UB:g} that,
\begin{align}\label{eq:Xn:2ndmom:exact}
\bE^\go[\cL_m]^2\,&= \, \gs^4\sum_{x,y\in \gL_m} g^m(0,x)^2 g^m(0,y)^2 \pi_\go(x)\pi_\go(y) \\
\notag 
&\leq \, c\,g^\go(0,0)^4 \sum_{|x|,|y|\leq N_2(0)} \pi_\go(x)\pi_\go(y) +  c\sum_{x,y\in \gL_m} |x|^{2(2-d)} |y|^{2(2-d)} \pi_\go(x)\pi_\go(y) \\
\notag &\qquad + 2\,c \left(\sum_{|x|\leq N_2(0)} \pi_\go(x)\right)\left(\sum_{y\in \gL_m} |y|^{2(2-d)}  g^\go(0,0)^2 \pi_\go(y) \right).
\end{align}
Let us take the expectation $\bbE$ of the upper bound. Assuming $p_0, \gz$ are large enough, one deduces from Claim~\ref{claim:pi:Lp} that the term $\sum_{|x|\leq N_2(0)} \pi_\go(x)$ is bounded in $L^p(\bbP)$ for $p\in(2,p_0)$. Assuming also that $p$ is large, this and~\eqref{eq:thm:AH21:greenLp} imply by H\"older's inequality that,
\[
\bbE\left[g^\go(0,0)^4 \sum_{|x|,|y|\leq N_2(0)} \pi_\go(x)\pi_\go(y)\right]\,<\,+\infty\,.
\]
Regarding the second term, it is dominated by $c\bE[\cL_m]^2$, since one deduces from the Cauchy-Schwarz inequality that $\bbE[\pi_\go(x)\pi_\go(y)]$ is bounded uniformly in $x,y\in\bbZ^d$. Furthermore, for $p\in(2,p_0)$ and $q\ceq \frac{p}{p-1}<2$, one has by Jensen's inequality,
\begin{align*}
&\bbE\left[\Bigg(\sum_{y\in \gL_m} |y|^{2(2-d)}g^\go(0,0)^2 \pi_\go(y) \Bigg)^q\right]^{1/q} \,\leq\, \bbE\left[\Bigg(\sum_{y\in \gL_m} |y|^{2(2-d)}g^\go(0,0)^2 \pi_\go(y)\Bigg)^2\right]^{1/2}\\
&\qquad \leq \left( \sum_{x,y\in \gL_m} |x|^{2(2-d)}|y|^{2(2-d)} \bbE\left[g^\go(0,0)^4\pi_\go(x)\pi_\go(y)\right] \right)^{1/2}.
\end{align*}
Again, $\bbE\left[g^\go(0,0)^4\pi_\go(x)\pi_\go(y)\right]$ is bounded uniformly in $x,y\in\bbZ^d$, provided that $p_0$ and $\gz$ are large enough in Theorem~\ref{thm:AH21:UB}. Hence the term above is bounded by $c\sum_{y\in\gL_m}|y|^{2(2-d)}$. Thus, taking the expectation of the third term in the upper bound~\eqref{eq:Xn:2ndmom:exact} and applying H\"older's inequality, it is lower than $c \bE[\cL_m]$ for some $c>0$. Finally, recalling Lemma~\ref{lem:Xn:1stmom}, this yields that 
\begin{equation}\label{eq:lem:Xn:2ndmom:result1}
\bbE[\bE^\go[\cL_m]^2]\,\leq\, c\,\bE[\cL_m]^2\,,
\end{equation} for some $c>0$ and all $m\in\N$, which is the expected result.
\smallskip

We now prove the first inequality in Lemma~\ref{lem:Xn:2ndmom}. 
Let us write $\bVar^\go[\cL_m]\ceq \bE^\go[\cL_m^2]-\bE^\go[\cL_m]^2$. Using the conditional variance decomposition with respect $(\ell^m(x))_{x\in\bbZ^d}$, one obtains,
\begin{align}\label{eq:Xn:Vardecompo}
\bVar^\go(\cL_m)&=  \frac1{\pi_\go(0)^2}\sum_{x\in\gL_m} \bE^\go[\ell^m(x)]\bVar^\go(L^{1,m}_{x,0}) \,+\, \frac1{\pi_\go(0)^2}\bVar^\go\Bigg(\sum_{x\in\gL_m} \ell^m(x) \bE^\go\big[L^{1,m}_{x,0}\big]\Bigg)\\
\notag & \eqc Y_m + Z_m.
\end{align}
We claim the following.
\begin{lemma}\label{lem:manytotwo}
There exists $c>0$ such that, for $x\in\bbZ^d$, $m\in\N$, and $\bbP$-a.e. $\go$,
\[
\bE^\go\big[(L^{1,m}_{x,0})^2\big] \,\leq\, c\sum_{y\in \gL_n} g^m(x,y) g^m(y,0)^2 \pi_\go(y)\pi_\go(0)^2\,.
\]
\end{lemma}
This lemma comes from a Many-to-two formula: we postpone its proof for now. By~\eqref{eq:lem:Xn:Lix:1stmom:n} and Lemma~\ref{lem:manytotwo}, one has,
\begin{equation}\label{eq:Xn:Vardecompo:Yn}
Y_m\,\leq \frac1{\pi_\go(0)^2} \sum_{x\in \gL_m}\bE^\go[\ell^m(x)]\bE^\go\big[(L^{1,m}_{x,0})^2\big]\leq c\sum_{x,y\in \gL_m} g^m(0,x) g^m(x,y) g^m(y,0)^2\pi_\go(x)\pi_\go(y),
\end{equation}
and,
\begin{align}
\notag Z_m &\leq \frac1{\pi_\go(0)^2}\bE^\go\Bigg[\bigg(\sum_{x\in \gL_m} \ell^m(x) \bE^\go\big[L^{1,m}_{x,0}\big]\bigg)^2\Bigg] \\
\notag &\leq \frac2{\pi_\go(0)^2}\sum_{x,y\in \gL_m} \bE^\go\left[ \sum_{\substack{u,v \in \Spine(T),\\u\prec v\text{ or }u=v}} \ind_{\{\cS_{T}(u)=x\}}\ind_{\{\cS_{T}(v)=y\}}\right] \bE^\go\big[L^{1,m}_{x,0}\big]\bE^\go\big[L^{1,m}_{y,0}\big]\\
\label{eq:Xn:Vardecompo:Zn}
&\leq c \sum_{x,y\in \gL_m} g^m(0,x)^2 g^m(x,y) g^m(y,0)\pi_\go(x)\pi_\go(y)\,.
\end{align}
Exchanging the notation $x$ and $y$ above, $Z_m$ and $Y_m$ have the same upper bound. Recall $\bE^\go[\cL_m]^2$ from~\eqref{eq:Xn:2ndmom:exact}: it suffices to show that the r.h.s. of~\eqref{eq:Xn:Vardecompo:Zn} is dominated by $c\bbE[\bE^\go[\cL_m]^2]$ for some $c>0$, then we conclude the proof of Lemma~\ref{lem:Xn:2ndmom} with~\eqref{eq:Xn:Vardecompo}.
\smallskip

Applying~\eqref{eq:green:rev} and the Cauchy-Schwarz inequality twice, the expectation of the r.h.s. of~\eqref{eq:Xn:Vardecompo:Zn} is bounded from above by,
\begin{align}\label{eq:Xn:2ndmom:Zn:1}
c\,\bbE\big[\bE^\go[\cL_m]^2\big]^{1/2}\bbE\Bigg[\sum_{x,y\in \gL_m} g^m(x,y)^2g^m(x,0)^2\pi_\go(x)\pi_\go(y)\Bigg]^{1/2}.
\end{align}
Moreover, $\bbP$ is invariant by the translation $\tau_{-x}$ for all $x\in\bbZ^d$, so one deduces that,
\begin{align}\label{eq:Xn:2ndmom:Zn:2}
\notag \bbE\left[\sum_{x,y\in \gL_m} g^m(x,y)^2 g^m(x,0)^2\pi_\go(x)\pi_\go(y)\right]\,
&\leq\, \bbE\left[\sum_{u,v\in \gL_{2m}} g^m(0,u)^2 g^m(0,v)^2\pi_\go(0)\pi_\go(u)\right]\\
\notag &=\, c\, \bbE\left[\bE^\go[\cL_{2m}]\,\pi_\go(0)\sum_{v\in \gL_{2m}} g^m(0,v)^2\right]\\
&\leq\, c\, \bE[\cL_{2m}]\, \bbE\left[\pi_\go(0)^2\left(\sum_{v\in \gL_{2m}} g^m(0,v)^2\right)^2\right]^{1/2}\,,
\end{align}
where we used~\eqref{eq:lem:Xn:1stmom}, the Cauchy-Schwarz inequality and~\eqref{eq:lem:Xn:2ndmom:result1}. Moreover, one deduces from~\eqref{eq:thm:AH21:UB:g} and the inequalities $g^m(0,v)\leq g^\go(0,v)\leq g^\go(0,0)$ that, for $\bbP$-a.e. $\go$,
\begin{equation}\label{eq:Xn:2ndmom:Zn:3}
0\,\leq\, \sum_{v\in \gL_{2m}} g^m(0,v)^2 \,\leq \, c\, g^\go(0,0)^2\, N_2(0)^d + c\sum_{v\in \gL_{2m}} |v|^{2(2-d)}\,.
\end{equation}
Plugging this into~\eqref{eq:Xn:2ndmom:Zn:2} and recalling from Lemma~\ref{lem:Xn:1stmom} that $\bE[\cL_{m}]\asymp \bE[\cL_{2m}]$, a direct computation yields 
that~\eqref{eq:Xn:2ndmom:Zn:2} is dominated by $c \bE[\cL_{m}]^2$.  Recollecting~\eqref{eq:Xn:2ndmom:Zn:1}, this finally yields that the r.h.s. of~\eqref{eq:Xn:Vardecompo:Zn} is bounded from above by the term $c\bbE\big[\bE^\go[\cL_m]^2\big]^{1/2} \bE[\cL_{m}]\leq c\bbE[\bE^\go[\cL_m]^2]$ for all $m\in\N$ (where we used Jensen's inequality), finishing the proof of Lemma~\ref{lem:Xn:2ndmom}.
\end{proof}

\begin{proof}[Proof of Lemma~\ref{lem:manytotwo}]
For a realization of the critical BRWRE $(\cS^1_x,\T^1_x)$, define $U\subset \T^1_x$ the set of vertices $u\in \T^1_x$ such that, for all $w\prec u$, one has $\cS^1_x(w)\in \gL_m$. Then,
\begin{align*}
\bE^\go\big[(L^{1,m}_{x,0})^2\big]= \bE^\go\Bigg[\bigg(\sum_{u\in U} \ind_{\{\cS^1_x(u)=0\}}\bigg)^2\Bigg]=\bE^\go[L^{1,m}_{x,0}] + \bE^\go\Bigg[\sum_{u,v\in U,u\neq v} \ind_{\{\cS^1_x(u)=0,\,\cS^1_x(v)=0\}}\Bigg].
\end{align*}
For $u,v\in\cT^1_x$, write $u\prec v$ if $u\neq v$ and $v$ is a descendant of $u$. On the one hand, one deduces from Markov's property that,
\[
\bE^\go\Bigg[\sum_{u,v\in U,u\prec v} \ind_{\{\cS^1_x(u)=0,\,\cS^1_x(v)=0\}}\Bigg]\,=\, \bE^\go\Bigg[\sum_{u\in U} \ind_{\{\cS^1_x(u)=0\}} \cdot \tilde L^{1,m}_{0,0}\Bigg]\,=\, \bE^\go\big[L^{1,m}_{x,0}\big]\bE^\go\big[L^{1,m}_{0,0}\big]\,,
\]
where $\tilde L^{1,m}_{0,0}$ is a copy of $L^{1,m}_{0,0}$, i.e. the local time in $0$ of some critical BRWRE $(\tilde\cS_0^1,\tilde \T^1_0)$ started from $0$ and constrained in $\gL_m$; and is taken independent from $(\cS^1_x,\T^1_x)$. On the other hand, for $u,v$ which are not on the same genealogical line, we let $w\ceq u\wedge v$ denote their most recent common ancestor and $c(w)$ its number of children in $\cT^1_x$. Then one deduces from Markov's property that,
\begin{align*}
&\bE^\go\Bigg[\sum_{w\in U} \,\sum_{\substack{u,v\in U,\\u\succ w, v\succ w,\\ u\wedge v=w}} \ind_{\{\cS^1_x(u)=0\}}\ind_{\{\cS^1_x(v)=0\}}\Bigg]\\
&\qquad=\, 
\bE^\go\Bigg[\sum_{y\in \gL_m}\sum_{w\in U} \ind_{\{\cS^1_x(w)=y\}} \,\sum_{\substack{u,v\in U,\\u\succ w, v\succ w,\\ u\wedge v=w}} \ind_{\{\cS^1_x(u)=0\}}\ind_{\{\cS^1_x(v)=0\}}\Bigg]\\
&\qquad= \,\bE^\go\Bigg[\sum_{y\in \gL_m}\sum_{w\in U} \ind_{\{\cS^1_x(w)=y\}}   \cdot  c(w)(c(w)-1) \cdot \tilde L^{1,m}_{y,0}  \cdot \tilde L^{2,m}_{y,0} \Bigg]\,,
\end{align*}
with $\tilde L^{i,m}_{y,0}$, $i\in\{1,2\}$, is the local time in $0$ of some independent BRWRE started from $y$ and constrained to $\gL_m$. Recalling that the root of $\cT^1_x\sim \PT$ reproduces as $((k+1)q_{k+1})_{k\geq0}$ and that $\bq$ has a finite third moment, one has
\begin{equation}\label{eq:useofthirdmoment}
\ET\big[c(w)^2\big]\,\leq\, \max\left(\gs^2+1,\sum_{k\geq1}k^3q_k\right)<+\infty\,.
\end{equation}
Recalling~\eqref{eq:lem:Xn:Lix:1stmom:n}, one finally obtains,
\begin{align*}
&\bE^\go\big[(L^{1,m}_{x,0})^2\big]\\
&\qquad\leq\, c\,\pi_\go(0) g^m(x,0) + 2 c\, \pi_\go(0)^2 g^m(x,0) g^m(0,0) + c \sum_{y\in \gL_m} \pi_\go(y)\pi_\go(0)^2 g^m(x,y) g^m(y,0)^2\\
&\qquad\leq\, c\sum_{y\in \gL_m}  g^m(x,y) g^m(y,0)^2 \pi_\go(y)\pi_\go(0)^2,
\end{align*}
where we used that $\pi_\go(0)g^m(0,0)=\bE^\go[L^{1,m}_{0,0}]\geq 1$.
\end{proof}

\subsection{Recurrence for traps}
When $d\leq 2$, recall that the BRWRE $(\cS_\Ti,\Ti)$ restricted to $\Spine(\cTi)$ has the same law as the RWRE $(X_n)_{n\geq0}$ in random traps, and that the latter visits the same vertices of $\bbZ^d$ as a simple random walk (recall e.g. the proof of Proposition~\ref{prop:green:traps}). Therefore, the critical snake is $\bP$-a.s. recurrent for $d\leq2$.

We now assume $d\in\{3,4\}$, that $\bbE \pi_\go(0) <+\infty$, and that Assumptions~\ref{assum:erg},~\ref{assum:traps} hold. Similarly to the random conductances environment, we proceed with a second moment method, however the truncation we use here is twofold: we kill particles when they leave the box $\gL_m$ for some $m\in\N$, or when they fall in a ``large trap''. Let $R>0$ and define, for $\bbP$-a.e. $\go$,
\begin{equation}
A_\go\,\ceq\, \{x\in\bbZ^d\,;\, \pi_\go(x) \geq R\,|x|^{2}\}\,.
\end{equation}
Let us consider the BRWRE indexed by $\Ti$ which is killed when it reaches \emph{either} $\partial \gL_m$ \emph{or} $A_\go$. Let $\cZ_m$ denote its renormalized local time in $0$: we shall prove that $\cZ_m$ satisfies the same annealed moment estimates than $\cL_m$ in Lemmata~\ref{lem:Xn:1stmom} and~\ref{lem:Xn:2ndmom}, then Theorem~\ref{thm:main:traps}.$(ii)$ follows from the same arguments as in the conductances case (i.e. the Paley-Zygmund inequality and Proposition~\ref{prop:01law}, we do not reproduce them).

\begin{lemma}\label{lem:Zn:1stmom}
Let $d\in\{3,4\}$. One has for some $c>0$ uniform in $m\in\N$,
\[
\bE[\cZ_m]\;\geq\; 
    \begin{cases}
    c\, m & \qquad\text{if }d=3,\\
    c\, \log m & \qquad\text{if }d=4.
    \end{cases}
\]
\end{lemma}

\begin{lemma}\label{lem:Zn:2ndmom} Let $d\in\{3,4\}$. One has for some $c>0$ uniform in $m\in\N$,
\[
\bE[\cZ_m^2]\;\leq\; 
    \begin{cases}
    c\, m^2 & \qquad\text{if }d=3,\\
    c\, (\log m)^2 & \qquad\text{if }d=4.
    \end{cases}
\]
\end{lemma}

Recall the definitions of $\ell^m(x)$, $L^{i,m}_{x,y}$ for the BRWRE killed outside at $\partial\gL_m$, and that~\eqref{eq:lem:Xn:Lix:1stmom:n} is unchanged in a random traps environment. We define similarly $\tilde\ell^m(x)$, $\tilde L^{i,m}_{x,y}$ for the BRWRE killed at $\partial\gL_m\cup A_\go$: in particular one has for $m\in\N$,
\begin{equation}
\cZ_m\,\ceq\,\frac1{\pi_\go(0)}\sum_{x\in \gL_m} \sum_{i=1}^{\tilde\ell^m(x)} \tilde L^{i,m}_{x,0}\,.
\end{equation}
Recall also from Proposition~\ref{prop:green:traps} that Green's function for the (non-killed) RWRE satisfies $g^\go(\cdot,\cdot)=g(\cdot,\cdot)$ $\bbP$-a.s.; and let $\tilde g^m$ denote Green's function of the (quenched) RWRE killed at $\partial\gL_m\cup A_\go$.

\begin{proof}[Proof of Lemma~\ref{lem:Zn:1stmom}]
Let $K>0$, $n\in\N$. Since one has $\pi_\go(x)\geq1$ $\bbP$-a.s. for all $x\in\bbZ^d$ in the random traps environment, one deduces from Jensen's inequality that,
\begin{align*}
\bE[\cZ_{Kn}]\,=\, \gs^2 \bbE\left[\sum_{x\in \gL_{Kn}} \tilde g^{Kn}(0,x)^2 \pi_\go(x)\right]\,\geq\, \gs^2 \sum_{n/2\leq |x|\leq n} \bbE\big[\tilde g^{Kn}(0,x)\big]^2 \,.
\end{align*}
Let us show that $\bbE[\tilde g^{Kn}(0,x)]\geq c|x|^{2-d}$ uniformly in $n/2\leq|x|\leq n$.
Recalling Lemma~\ref{lem:constrainedgreen}, there exists $c_1>0$ such that for $n/2\leq|x|\leq n$, one has $\bbP$-a.s.,
\[
c_1\,|x|^{2-d}\,\leq\,g^{Kn}(0,x)\,\leq\, \tilde g^{Kn}(0,x) + \sum_{y\in \gL_{Kn}} g^\go(0,y)g^\go(y,x)\ind_{\{y\in A_\go\}}\,,
\]
Moreover, the stationarity of $\go$ yields $\bbP(y\in A_\go)=\bbP(\pi_\go(0)\geq R\,|y|^2)$. 
Taking the expectation $\bbE$ above, and using Proposition~\ref{prop:green:traps} and~\eqref{eq:UE:green}, we claim that it is enough to show that
\begin{equation}\label{eq:lem:Zn:1stmom:1}
\sum_{y\in \gL_{Kn}} |y|^{2-d}|x-y|^{2-d}\,\bbP(\pi_\go(0)\geq R\,|y|^2)\,\leq\, \frac{c_1}2 |x|^{2-d}\,,
\end{equation}
and the result follows. On the one hand we have by Markov's inequality,
\begin{equation*}
\sum_{y:|x-y|\leq n/4} |y|^{2-d}|x-y|^{2-d}\,\bbP\big(\pi_\go(0)\geq R|y|^2\big)\,\leq\, \frac{c}{R}\, n^{-d} \sum_{y:|x-y|\leq n/4} |x-y|^{2-d}\,\leq\, \frac{c}{R} n^{2-d}\,,
\end{equation*}
for some $c>0$ that does not depend on $R$. On the other hand,
\begin{equation*}
\sum_{y:|x-y|\geq n/4} |y|^{2-d}|x-y|^{2-d}\,\bbP\big(\pi_\go(0)\geq R|y|^2\big)\,\leq\, c\, n^{2-d} \sum_{y:|x-y|\geq n/4} |y|^{2-d}\bbP\big(\pi_\go(0)\geq R|y|^2\big)\,,
\end{equation*}
and
\[
\sum_{y:|x-y|\geq n/4} |y|^{2-d}\bbP\big(\pi_\go(0)\geq R|y|^2\big) \,\leq\, c\int_0^{+\infty} r\,\bbP\big(\pi_\go(0)\geq R\,r^2\big) \, \dd r\,\leq\, \frac{c}R\,,
\]
where we used that $\bbE[\pi_\go(0)]<+\infty$. Assuming $R$ is large enough, this finally yields~\eqref{eq:lem:Zn:1stmom:1}.
\end{proof}

\begin{proof}[Proof of Lemma~\ref{lem:Zn:2ndmom}]
Recall~(\ref{eq:Xn:Vardecompo}--\ref{eq:Xn:Vardecompo:Zn}): reproducing this computation with $\cZ_m$ and using Proposition~\ref{prop:green:traps}, this yields $\bbP$-a.s.,
\begin{align*}
\bVar^\go(\cZ_m)&= \bE^\go[\cZ_m^2]-\bE^\go[\cZ_m]^2 \\
&\leq c\sum_{x,y\in \gL_m} g(0,x)^2g(x,y)g(0,y) \pi_\go(x)\pi_\go(y)\ind_{\{\pi_\go(x)\leq R|x|^2,\pi_\go(y)\leq R|y|^2\}},\\
\text{and}\qquad \bE^\go[\cZ_m]^2&\leq c\sum_{x,y\in \gL_m} g(0,x)^2g(0,y)^2\pi_\go(x)\pi_\go(y)\ind_{\{\pi_\go(x)\leq R|x|^2,\pi_\go(y)\leq R|y|^2\}}.
\end{align*}
Recall that, under Assumption~\ref{assum:traps}, there exists $R_0,K>0$ such that for $x,y$ with $|x-y|\geq R_0$, one has $\bbE[\pi_\go(x)\pi_\go(y)]\leq K$. Moreover, one has for $x,y\in\bbZ^d$, 
\[
\bbE[\pi_\go(x)\pi_\go(y)\ind_{\{\pi_\go(x)\leq R|x|^2,\pi_\go(y)\leq R|y|^2\}}]\,\leq\, R|x|^2\bbE[\pi_\go(y)]\,\leq\, c|x|^2\,.
\]
Thus,
\begin{align*}
\bbE\big[\bE^\go[\cZ_m]^2\big]\,&\leq\, c\sum_{x,y\in \gL_m} g(0,x)^2g(0,y)^2\left[K\ind_{|x-y|\geq R_0} + c|x|^2 \ind_{|x-y|<R_0}\right]\\
&\leq\, c\sum_{x,y\in \gL_m} |x|^{2(2-d)}|y|^{2(2-d)} + c\sum_{x\in \gL_m} |x|^{4(2-d)+2}\,.
\end{align*}
Using standard series estimates, this yields the upper bound expected in Lemma~\ref{lem:Zn:2ndmom}. Similarly, one obtains
\[
\bbE\bVar^\go(\cZ_m)\,\leq\, c\sum_{x,y\in \gL_m} |x|^{2(2-d)}|x-y|^{2-d}|y|^{2-d} + c\sum_{x\in \gL_m} |x|^{3(2-d)+2}\,.
\]
Using the Cauchy-Schwarz inequality and a shift-invariance argument similarly to~(\ref{eq:Xn:2ndmom:Zn:1}-\ref{eq:Xn:2ndmom:Zn:2}), this yields the same upper bound (we leave the details to the reader). Recalling that $\bE[\cZ_m^2]= \bbE\bVar^\go(\cZ_m) + \bbE[\bE^\go[\cZ_m]^2]$, this completes the proof.
\end{proof}


\appendix
\section{Proof of Proposition~\ref{prop:rwre:rec}}\label{app:dim2}
This is a standard argument, which follows from the well-known \emph{Dirichlet principle}. We present the main ideas of the proof here, and refer to~\cite[Proposition~3.38]{AF02} (or~\cite{DS84, LP16} among other references) for more details. For $m\in\N$ and $f:\gL_{m+1}\to\R$ such that $f(y)=0$ for all $y\in\partial\gL_m$, its associated \emph{Dirichlet energy} in $\go\in\gO$ is defined by,
\begin{equation}
\cE^\go_m(f)\,\ceq\,\frac12\sum_{x,y\in\gL_m} \go_{x,y} [f(y)-f(x)]^2\,,
\end{equation}
and we write $\cE^1_m$ for the Dirichlet energy in homogeneous environment, i.e. $\go_{x,y}\ceq\ind_{x\sim y}$ almost surely for $x,y\in\bbZ^d$. Then the \emph{effective conductance} $C^\go(0,\partial\gL_m)$ between $0$ and $\partial\gL_m$ is defined by,
\begin{equation}\label{eq:def:effconduct}
C^\go(0,\partial\gL_m)\,\ceq\, \inf\left\{\cE^\go_m(f)\,;\, f:\gL_{m+1}\to\R, f(0)=1, f(y)=0\;\forall\, y\in\partial\gL_m\right\}\,,
\end{equation}
and, letting $\tau^+_A\ceq\inf\{n\geq1, X_n\in A\}$ the return time to any set $A\subset\gL_{m+1}$, one has,
\begin{equation}
C^\go(0,\partial\gL_m)\,=\, \pi_\go(0) \bP^\go(\tau^+_{\{0\}} > \tau^+_{\partial\gL_m})\,.
\end{equation}
Moreover, the infimum~\eqref{eq:def:effconduct} for $\cE^\go_m$ (resp. $\cE^1_m$) is achieved by some harmonic function $h^\go$ (resp. $h$). 

It follows from the Dirichlet principle that, for $\bbP$-a.e. $\go$, one has
\begin{align*}
\pi_\go(0) \bP^\go(\tau^+_{\{0\}} > \tau^+_{\partial\gL_m}) \,=\, \cE^\go_m(h^\go) \,\leq\, \cE^\go_m(h)\,=\, \frac12\sum_{x,y\in\gL_m} \go_{x,y} [h(y)-h(x)]^2\,.
\end{align*}
Taking the expectation above, this yields,
\begin{align*}
\bbE\left[\pi_\go(0) \bP^\go(\tau^+_{\{0\}} > \tau^+_{\partial\gL_m})\right]\,\leq\, \sup_{x\sim y}\bbE[\go_{x,y}] \times \cE^1_m(h)\,\leq\, c \, \tilde\bP(\tau^+_{\{0\}} > \tau^+_{\partial\gL_m})\,,
\end{align*}
where $\tilde\bP$ denotes the law of the homogeneous random walk in $\bbZ^d$. When $d\leq2$ the latter goes to 0 as $m\to+\infty$ (since the homogeneous random walk is recurrent), so this yields that $\bP^\go(\tau^+_{\{0\}} =+\infty)=0$ for $\bbP$-a.e. $\go$, finishing the proof.
\qed

\bibliographystyle{abbrv}
\bibliography{biblio}

\begin{thebibliography}{10}

\bibitem{AF02}
D.~Aldous and J.~A. Fill.
\newblock Reversible markov chains and random walks on graphs, 2002.
\newblock Unfinished monograph, recompiled 2014, available at
  \url{http://www.stat.berkeley.edu/$\sim$aldous/RWG/book.html}.

\bibitem{ADS16}
S.~Andres, J.-D. Deuschel, and M.~Slowik.
\newblock Harnack inequalities on weighted graphs and some applications to the
  random conductance model.
\newblock {\em Probability Theory and Related Fields}, 164(3):931--977, Apr
  2016.

\bibitem{AH21}
S.~Andres and N.~Halberstam.
\newblock Lower {G}aussian heat kernel bounds for the random conductance model
  in a degenerate ergodic environment.
\newblock {\em Stochastic Process. Appl.}, 139:212--228, 2021.

\bibitem{BAC06}
G.~B. Arous and J.~Černý.
\newblock Course 8 - dynamics of trap models.
\newblock In A.~Bovier, F.~Dunlop, A.~{van Enter}, F.~{den Hollander}, and
  J.~Dalibard, editors, {\em Mathematical Statistical Physics}, volume~83 of
  {\em Les Houches}, pages 331--394. Elsevier, 2006.

\bibitem{BC12}
I.~Benjamini and N.~Curien.
\newblock {Recurrence of the $\mathbb{Z}^d$-valued infinite snake via
  unimodularity}.
\newblock {\em Electronic Communications in Probability}, 17:1--10, 2012.

\bibitem{BP94}
I.~Benjamini and Y.~Peres.
\newblock Markov chains indexed by trees.
\newblock {\em Ann. Probab.}, 22(1):219--243, 1994.

\bibitem{BP94b}
I.~Benjamini and Y.~Peres.
\newblock Tree-indexed random walks on groups and first passage percolation.
\newblock {\em Probab. Theory Related Fields}, 98(1):91--112, 1994.

\bibitem{BBHK08}
N.~Berger, M.~Biskup, C.~E. Hoffman, and G.~Kozma.
\newblock Anomalous heat-kernel decay for random walk among bounded random
  conductances.
\newblock {\em Ann. Inst. Henri Poincar\'{e} Probab. Stat.}, 44(2):374--392,
  2008.

\bibitem{Big04}
J.~D. Biggins and A.~E. Kyprianou.
\newblock Measure change in multitype branching.
\newblock {\em Adv. in Appl. Probab.}, 36(2):544--581, 2004.

\bibitem{Bis11Rev}
M.~Biskup.
\newblock {Recent progress on the Random Conductance Model}.
\newblock {\em Probability Surveys}, 8:294 -- 373, 2011.

\bibitem{BS12book}
E.~Bolthausen and A.~Sznitman.
\newblock {\em Ten Lectures on Random Media}.
\newblock Oberwolfach Seminars. Birkh{\"a}user Basel, 2012.

\bibitem{Bou92}
J.~Bouchaud.
\newblock {Weak ergodicity breaking and aging in disordered systems}.
\newblock {\em {Journal de Physique I}}, 2(9):1705--1713, 1992.

\bibitem{Chau91}
B.~Chauvin.
\newblock Product martingales and stopping lines for branching {B}rownian
  motion.
\newblock {\em Ann. Probab.}, 19(3):1195--1205, 1991.

\bibitem{CP07}
F.~Comets and S.~Popov.
\newblock On multidimensional branching random walks in random environment.
\newblock {\em The Annals of Probability}, 35(1):68--114, 2007.

\bibitem{Del99}
T.~Delmotte.
\newblock Parabolic {H}arnack inequality and estimates of {M}arkov chains on
  graphs.
\newblock {\em Rev. Mat. Iberoamericana}, 15(1):181--232, 1999.

\bibitem{DS84}
P.~G. Doyle and J.~L. Snell.
\newblock {\em Random walks and electric networks}, volume~22 of {\em Carus
  Mathematical Monographs}.
\newblock Mathematical Association of America, Washington, DC, 1984.

\bibitem{Dur03}
B.~Durhuus.
\newblock Probabilistic aspects of infinite trees and surfaces.
\newblock {\em Acta Physica Polonica B}, 2003.

\bibitem{GantertMuller}
N.~Gantert and S.~M\"{u}ller.
\newblock The critical branching {M}arkov chain is transient.
\newblock {\em Markov Process. Related Fields}, 12(4):805--814, 2006.

\bibitem{Jag89}
P.~Jagers.
\newblock General branching processes as {M}arkov fields.
\newblock {\em Stochastic Process. Appl.}, 32(2):183--212, 1989.

\bibitem{Kes86}
H.~Kesten.
\newblock Subdiffusive behavior of random walk on a random cluster.
\newblock {\em Annales de l'I.H.P. Probabilit\'es et statistiques},
  22(4):425--487, 1986.

\bibitem{Kesten95}
H.~Kesten.
\newblock Branching random walk with a critical branching part.
\newblock {\em J. Theoret. Probab.}, 8(4):921--962, 1995.

\bibitem{LGL16}
J.-F. Le~Gall and S.~Lin.
\newblock The range of tree-indexed random walk.
\newblock {\em J. Inst. Math. Jussieu}, 15(2):271–317, 2016.

\bibitem{LP16}
R.~Lyons and Y.~Peres.
\newblock {\em Probability on Trees and Networks}, volume~42 of {\em Cambridge
  Series in Statistical and Probabilistic Mathematics}.
\newblock Cambridge University Press, New York, 2016.
\newblock Available at \url{https://rdlyons.pages.iu.edu/}.

\bibitem{Muller08}
S.~M\"{u}ller.
\newblock A criterion for transience of multidimensional branching random walk
  in random environment.
\newblock {\em Electron. J. Probab.}, 13:no. 41, 1189--1202, 2008.

\end{thebibliography}

\section*{Acknowledgments}
The authors acknowledge support from the grant ANR-22-CE40-0012 (project Local). C.S. is also supported by the Institut Universitaire de France. 

\end{document}